\documentclass[a4paper,11pt]{article}

\usepackage{amsmath,amssymb,amsthm}

\usepackage{amsmath,  amssymb}

\usepackage{amsfonts,amssymb}

\numberwithin{equation}{section}

\newtheorem{theorem}{Theorem}[section]

\newtheorem{lemma}[theorem]{Lemma}

\newtheorem{remark}[theorem]{Remark}
\newtheorem{corollary}[theorem]{Corollary}
\newtheorem{definition}[theorem]{Definition}

\allowdisplaybreaks

\begin{document}
\title{On the Entire Radial Solutions of the {C}hern-{S}imons SU$(3)$ System}
\author{Hsin-Yuan Huang\,$^1$ and Chang-Shou Lin\,$^2$}
\date{\small\it $^1$Department of Applied Mathematics, National Sun Yat-sen University,\\ Kaoshiung 804, Taiwan\vspace{.3cm}\\
$^2$Taida Institute for Mathematical Sciences, Center for Advanced Study in Theoretical Science,     \\National Taiwan University, Taipei, 106, Taiwan}
\maketitle
\begin{abstract}

In this paper, we study the entire radial  solutions of the self-dual equations arising from the relativistic  {${\rm SU}(3)$} Chern-Simons model proposed by Kao-Lee\cite{KL} and  Dunne\cite{D1,D2}. Understanding the structure of entire radial solutions is one of  fundamental issues for the system of nonlinear equations. In this paper, we prove any entire radial solutions must be one of  topological, non-topological and mixed type solutions, and completely classify the asymptotic behaviors at infinity of these solutions.
Even for radial solutions, this classification has remained an open problem for many years. As an application of this classification, we prove that the two components $u$ and $v$ have intersection at most finite times.

\end{abstract}

\section{Introduction}

The relativistic self-dual Abelian Chern-Simons  model was proposed by Jackiw-Weinberg\cite{JW} and Hong-Kim-Pac\cite{HKP} to study the physics of high critical temperature super conductivity. The corresponding Chern-Simons equation has been studied  in a variety of different nature.   We refer the reader to Wang\cite{Wang}, Spruck-Yang\cite{SY}, Caffarelli-Yang\cite{CY}, Tarantello \cite{T1}, Chae-Imanuvilov\cite{CI}, Nolasco-Tarantello \cite{NT} \cite{NT1}, Chan-Fu-Lin\cite{CFL}, Choe\cite{choe}, Lin-Yan\cite{LY} and Choe-Kim-Lin\cite{CKL} for the recent developments.\par

In this paper, we are interested in the non-Abelian Chern-Simons  model  proposed by Kao-Lee  \cite{KL}  and Dunne\cite{D1,D2}. This model is defined in the $(2+1)$ Minkowski space ${\mathbb R}^{1,2}$, and the gauge group is a compact Lie group with a semi-simple Lie algebra $\mathcal{G}$. The Chern-Simons Lagrangian density in  $2+1$ dimensional spacetime involves the Higgs field $\phi$  and the gauge potential $A=(A_0,A_1,A_2)$.
We restrict to consider the energy minimizers of the Lagrangian functional, and thus obtain the self-dual Chern-Simons equations
\begin{equation}\label{e001}
\begin{split}
&D_{-}\phi=0\\
&F_{+-}=\frac{1}{k^2}\Big[\phi-[[\phi,\phi^{\dag}],\phi], \phi^{\dag}\Big]
\end{split}
\end{equation}
where $D_{-}=D_1-iD_2$, and $F_{+-}=\partial_{+}A_{-}-\partial_{-}A_{+}+[A_{+},A_{-}]$
with $A_{\pm}=A_{1}\pm iA_{2}$ and $\partial{\pm}=\partial_1\pm i\partial_2$.
Dunne considered a simplified form of the self-dual system \eqref{e001},
in which the fields $\phi$ and $A$ are algebraically restricted:
$$\phi=\sum_{a=1}^r\phi^{a}E_a$$
where $r$ is the rank of the gauge Lie algebra,  $E_a$ is the simple root step operator, and $\phi^{a}$ are complex-valued functions. Let
$$u_a=\log|\phi^a|,\,\,a=1,2,\cdots,r. $$
Then equations \eqref{e001} can be reduced to the following system of  equations
\begin{equation}\label{e013}
\Delta u_a+\frac{1}{k^2}\Big(\sum_{b=1}^rK_{ab}e^{u_b}-\sum_{b=1}^r e^{u_b}K_{bc}e^{u_c}K_{ac}\Big)=4\pi\sum_{j=1}^{N_a}\delta_{p_j^a}  , \,\,\,a=1,\cdots,r,
\end{equation}
where $K=(K_{ab})$ is the Cartan matrix of a semi-simple Lie algebra, $\{p_j^a\}_{j=1,\cdots, N_a}$ are zeros of $\phi^a$ ($a=1,\cdots r$), and $\delta_{p}$ is the Dirac measure concentrated at $p$ in $\mathbb{R}^2$.
We refer to \cite{Yang2}  for the details from \eqref{e001} to \eqref{e013}.\par
Suppose $K$ satisfies $\sum_{b=1}^r(K^{-1})_{ab}>0$, $a=1,\cdots r$. A solution of \eqref{e013}, $(u_a)_{a=1,\cdots,r}$, verifying the asymptotic condition $$\hspace{2cm}\lim_{|x|\to\infty}u_a(x)=\log \Big(\sum_{b=1}^r(K^{-1})_{ab}\Big),\,\,a=1,\cdots, r, \leqno(I)$$
is called a {\bf topological solution}; a solution of \eqref{e013}, $(u_a)_{a=1,\cdots,r}$, verifying the asymptotic condition
$$\hspace{2cm}\lim_{|x|\to\infty}u_a(x)=-\infty,\,\,a=1,\cdots, r,\leqno(II) $$
is called a {\bf non-topological solution}.  Yang in \cite{Yang1} obtained the existence of topological solutions for \eqref{e013} in $\mathbb{R}^2$
based on variational methods and a Cholesky decomposition technique.\par

%
In this paper, we consider the case when the gauge group is $SU(3)$ and the corresponding Cartan matrix
$K=\left(\begin{array}{cc}
2 & -1 \\
-1 & 2 \\
\end{array}\right).$
See \cite{NT,LY1,ALW} for the recent developments.\par

Suppose there is only one vortex at origin. Then the equations \eqref{e013} become
\begin{equation}\label{QAZ}
\left\{
\begin{split}
\Delta u&=-2e^u+e^v +4e^{2u}-2e^{2v}-e^{u+v} +4\pi N_1\delta_0\\
\Delta v&=e^u-2e^v -2e^{2u}  +4e^{2v}-e^{u+v}+4\pi N_2\delta_0
\end{split}
\hspace{1cm}\text{ in }\,\mathbb{R}^2.
\right.
\end{equation}
Here  $(u,v)=(u_1,u_2)$, $N_1$, $N_2\geq 0$, and without loss of generality, we assume $k=1$.\par
The purpose of this paper is to study the asymptotic behaviors at infinity of the entire radial solutions to \eqref{QAZ}.
Let $r=|x|$. Due to the singularity at $0$, $u$ and $v$ are assumed to satisfy
\begin{equation}\left\{
\begin{array}{l}
u(r)=2N_1 \log r+O(1),\\
v(r)=2N_2 \log r+O(1),
\end{array}\right.
\mbox{  as  }r\to 0^+.
\end{equation}
Conventionally, an entire solution might be classified as topological or non-topological solution according to its boundary condition at $\infty$. However, there might be another new type of solution, {\bf mixed-type solution}:
$$
 \begin{array}{c}
 (u(r),v(r))\to(\log\frac{1}{2},-\infty)\mbox{  as  }{r\to\infty}.\\
                          \mbox{    and }\\
                               (u(r),v(r))\to(-\infty, \log\frac{1}{2})\mbox{  as  }{r\to\infty}.\end{array}\leqno(III)
$$
We note that this type is new. The existence of the mixed-type solution has been asked by Nolasco-Tarantello\cite{NT} and this problem has remained an open problem since then.  One of the purposes in this series of paper is to answer this question.  Our first result
is the following:
\begin{theorem}\label{thm2}
Let  $(u(r),v(r))$ be an entire radial  solutions of \eqref{QAZ}. Then
 $$u(r)< 0\mbox{   and  }v(r)< 0\mbox{   for    }r\in(0,\infty),$$
  unless $u(r)=v(r)=0$ for $r\in(0,\infty)$.
\end{theorem}
Our second result is the main result of this paper.
\begin{theorem}\label{cor1}
Suppose $(u(r),v(r))$ is an entire solution to equations \eqref{QAZ}. Then $(u(r),v(r))$ must be one of  types (I), (II) and (III)  described above.
\end{theorem}
Very recently, some existence of non-topological solutions has been studied by Ao-Lin-Wei\cite{ALW} by a perturbation from the SU(3) Toda system with singular sources. However, their result is still very limited toward understanding the general theory of non-topological solutions (and mixed-type solutions). Our study of radial entire solution would play a significant role for this purpose.
The classification in Theorem \ref{cor1} is the first major result in this direction.\par

When $u=v$, equation \eqref{QAZ} is reduced to
\begin{equation}\label{e014}
\Delta u+e^{u}(1-e^u)=4\pi N\delta_0.
\end{equation}
For equation \eqref{e014}, it is easy to see that Theorem \ref{cor1} holds for any solution $u$ of \eqref{e014}, i.e. either $u$ is a topological solution or a non-topological solution. Obviously, this statement is equivalent to the claim:
\begin{center}
 for any solution $u$ to \eqref{e014}, $e^u(1-e^u)\in L^1(\mathbb{R}^2)$.
\end{center}
From equation \eqref{e014}, the $L^1$-integrability of $e^u(1-e^u)$ (which is always positive) can be easily obtained by integrating \eqref{e014}.
However, it is not obvious at all that the $L^1$-integrability of the nonlinear terms in equation \eqref{QAZ} holds. In fact, it is not clear whether both nonlinear terms in the right hand side of \eqref{QAZ} are positive for $r>0$.\par
Our proof of Theorem \ref{cor1} is based on the following observation. We split the nonlinear terms in (\ref{QAZ}) into linear combination of $f_1$ and $f_2$, where
$$f_1(u,v)=e^u-2e^{2u}+e^{u+v}\text{    and   }f_2(u,v)=e^{v}-2e^{2v}+e^{u+v}. $$
 Then \eqref{QAZ} becomes
\begin{equation}\label{QWE}
\left\{
\begin{split}
&u_{rr}+\frac{1}{r}u_r=f_2-2f_1+4\pi N_1\delta_0,\\
&v_{rr}+\frac{1}{r}v_r=f_1-2f_2+4\pi N_2\delta_0,
\end{split}
\right.
\,\,r\in(0,\infty).
\end{equation}
For convenience, we denote $f_i(r)=f_i(u(r),v(r))$, $i=1,2$. We observe that $f_1(r)$ and $f_2(r)$ might be positive functions in $(0,\infty)$. Note that if $u>v$(resp. $v>u$), then $f_2(u,v)>0$(resp. $f_1(u,v)$) automatically. So the question is whether $f_i(u,v)$($i=1,2$) is positive or not. We believe that the positivity of both $f_1$  and $f_2$, will play an important role for the further study of uniqueness of solutions of the system \eqref{QAZ}.
As an application of this positivity, we prove the following apriori estimate for not topological solutions of \eqref{QWE}.
\begin{theorem}\label{thm1}
Let $(u,v)$ be an entire radial solution of \eqref{QAZ}. Then

\begin{itemize}
\item[(1)] There exists $R_0\geq 0$ such that $u(r)$ and $v(r)$ are less than $\log\frac{1}{2}$ for $r\geq R_0$.
\item[(2)] $(u,v)$  is a topological solution. Furthermore, 
\begin{itemize}
\item[(a)]  $(u,v)$ is a topological solution if and only if $r(u+v)(r)_r >0$ on $(0,\infty)$.\\
\item[(b)]  Two components $u$ and $v$ have no intersection on $(R_0,\infty)$ and are increasing to $0$  as $r\to\infty$.
\end{itemize}
\end{itemize}
\end{theorem}
In particular, when $N_1=N_2=0$, we have the uniqueness of topological solutions.
\begin{corollary}
Suppose (u,v) is a topological solution of \eqref{QWE} with $N_1=N_2=0$, then $(u(r),v(r)\equiv(0,0)$.
\end{corollary}
The most difficult part of Theorem \ref{thm1} is the part (1), where the a priori bound $\log\frac{1}{2}$ is established. The proof of Theorem \ref{thm1} is unusually long. One of difficulties is to exclude the possibility of intersection of infinite times between $u$ and $v$.  From equation \eqref{QAZ}, this exclusion is not obvious at all, in fact, it is one of the consequences of $L^1$-integrability of $f_i$, $i=1,2$. After establishing Theorem \ref{thm1}, $L^1$-integrability of $f_i$, $i=1,2$, is a nice application of the Pohozaev identity. Then Theorem \ref{cor1} follows immediately. \par

\begin{corollary}\label{cor2}
Let $(u(r),v(r))$ be an entire radial solution of \eqref{QAZ}. Then the followings hold.
\begin{itemize}
\item[(1)] If $(u(r),v(r))$ is a non-topological solution, then
$$u(r)=-\beta_1\log r +O(1)$$
$$v(r)=-\beta_2\log r+O(1)$$
at $\infty$ for some $\beta_1>2$ and $\beta_2>2$. Furthermore,
$$\beta_1^2+\beta_1\beta_2+\beta_2^2-4(N_1^2+N_1N_2+N_2^2)>6\Big( 2(N_1+N_2)+\beta_1+\beta_2   \Big).   $$
\item[(2)] $(u(r),v(r))$ is a mixed-type solution, then either \\
$$   u(r)\to\log\frac{1}{2}\mbox{    and    }v(r)=-\beta\log r+O(1)\mbox{   for  some   }\beta>2,              $$
or
$$   v(r)\to\log\frac{1}{2}\mbox{   and    }u(r)=-\beta\log r+O(1)\mbox{   for  some   }\beta>2,              $$
as $r\to\infty$.\par
\item[(3)] If $(u(r),v(r))$ is a topological solution, then
$$(u(r),v(r))\to (0,0)\mbox{  exponentially  as    }r\to\infty. $$
\end{itemize}
\end{corollary}
By this corollary,  when $(u,v)$ is a non-topological or mixed-type solution, we have the positivity of $f_i(r)$, $i=1,2$, for $r$ large enough.
\begin{corollary}\label{cor3}
Suppose $(u(r),v(r))$ be a non-topological or mixed-type solution of \eqref{QAZ}. Then \\
$f_i(r)>0$, $i=1,2$, for $r$ large enough.
\end{corollary}

By Theorem \ref{thm1} and Corollary \ref{cor2}, we obtain:
\begin{corollary}
Let $(u(r),v(r))$ be an entire radial solution of \eqref{QAZ}. Then $u$ and $v$ have intersection finite times.
\end{corollary}

To appreciate the  result of Theorem \ref{cor1}, we should compare it with the following system of equations.
\begin{equation}\label{e004}
\left\{
\begin{split}
&\Delta u+e^v(1-e^u)=4\pi N_1\delta_0\\
&\Delta v+e^u(1-e^v)=4\pi N_2\delta_0
\end{split}
\right.
\hspace{2cm}\text{ in }\,\mathbb{R}^2,
\end{equation}
The system \eqref{e004} is related to Chern-Simons-Higgs model with  two Higgs particles. See \cite{LPY,CCL}. In spite of simple nature of the nonlinear terms in \eqref{e004}, Corollary \ref{cor2} does not hold for all solutions to \eqref{e004}. See \cite{CCL} for more details of statements.\par
Next, we want to consider the existence problem of mixed-type solutions. We denote $\Big(u(r;\alpha_1,\alpha_2), v(r;\alpha_1,\alpha_2)\Big)$ be  a radial solution of \eqref{QAZ} with
the initial value
\begin{equation}
\left\{
\begin{array}{l}
u(r ) = 2N_1 \log r +\alpha_1 + o(1)\\
v(r ) = 2N_2 \log r +\alpha_2 + o(1)
\end{array}
\right.
\text{   as  }  r  \to 0^+.
\end{equation}

We define the region of initial data of the non-topological solutions of \eqref{QAZ}.
\begin{equation}
\begin{split}
\Omega=\{(\alpha_1,\alpha_2)|&(u(r;\alpha_1,\alpha_2),v(r;\alpha_1,\alpha_2))\text{ is a  non-topological solution of }(\ref{QAZ}) \}.
\end{split}
\end{equation}

\begin{theorem}\label{thm5}
 $\Omega$ is an open set. Furthermore, if $\alpha=(\alpha_1,\alpha_2)\in\partial\Omega$, then $u(r;\alpha)$ is either a topological solution or a mixed-type solution.
\end{theorem}
In this paper, we do not address the question whether $\Omega$ is an non-empty. In a forthcoming paper, we will discuss this question completely.  However, for $N_1=N_2$, it is clear that by letting $u=v$, solutions of \eqref{e014} give $\Omega\not = \emptyset$. In fact, as a corollary of the existence result in \cite{ALW}, we have  $\Omega\not = \emptyset$, for all $N_1$ and $N_2$.\par
It is not difficult to prove that $\Omega\not=\mathbb{R}^2$, hence, $\partial\Omega\not=\emptyset$. It is generally expected that \eqref{QWE} should possess an unique topological solution. When $N_1=N_2=0$, we have the uniqueness of topological solutions.  We shall study this uniqueness problem in our further works. By this uniqueness, we should be able to prove the existence of the mixed-type solutions. In fact, we have the following conjecture.\\
\\
{\bf Conjecture} For any $\beta>2$, there is mixed-type solution of \eqref{QAZ} such that $u(r)\to\log\frac{1}{2}$, and $v(r)=-\beta\log r+O(1)$ as $r\to\infty$.\\
\par
The paper is organized as follows.  Sect. 2 is devoted to the proof of Theorem \ref{thm2}. In Sects. 3, 4 and 5, we show some  apriori estimates on the behavior of solutions. Theorem \ref{thm1} is proved in Sect. 6. The integrability of $f_1$ and $f_2$ will be discussed in Sect. 7 and Theorem \ref{cor1} follows. The asymptotic behaviors at infinity of solutions are shown in  Sect. 8.  Finally, in Sect. 9, we discuss the structure of the non-topological  solutions.

\section{PROOF OF THEOREM \ref{thm2}}

We introduce the function $g(x)=e^{x}-2e^{2x}$, which has the following property:\\
\begin{itemize}
\item[(1)]$g$\mbox{   is increasing on    }$(-\infty, \log\frac{1}{4})$, $g'(\log\frac{1}{4})=0$,  and $g$  is decreasing on $(\log\frac{1}{4},\infty).$ ( Hence, $g(x)\leq \frac{1}{8}$ on $(-\infty,\infty)$.)
\item[(2)] $g>0$ on $(-\infty,\log\frac{1}{2})$ and $g<0$ on $(\log\frac{1}{2},\infty)$.
\item[(3)] If $x<y$  and   $g(x)-g(y)=0$,  then   $x<\log\frac{1}{4}<y<\log\frac{1}{2}.$
\end{itemize}

The property of $g$ will be used in the lemmas of this paper.\par

If  $u(r_0)=v(r_0)$ and $u_r(r_0)=v_r(r_0)$, by the uniqueness,  $u(r)=v(r)$ for $r>0$. The system of  equations \eqref{QAZ} can be reduce to the single equation $$u_{rr}+\frac{1}{r}u_r= e^{2u}-e^{u}\mbox{  on  }(0,\infty).$$
Then it is known that $\lim_{r\to\infty}(u(r),v(r))=(0,0)$ or $\lim_{r\to\infty}(u(r),v(r))=(-\infty,-\infty)$ (see \cite{Yang2}).
Hence, if $u_r(r_0)=v_r(r_0)$, we can assume $u(r_0)\not = v(r_0)$; if $u(r_0)=v(r_0)$, we can assume  $u_r(r_0)\not = v_r(r_0)$.\par
We show some sufficient conditions that $(u,v)$ cannot be entire  solutions.
\begin{lemma}(Finite Time Blow-up Condition)\label{blowup}
\begin{itemize}
\item[(1)]  Suppose $u(r_0)=v(r_0)>0$, $u_r(r_0)>0$ and $(u-v)_r(r_0)>0$.  Then $u(r)$  blows up in finite time.
\item[(2)]  Suppose $u(r_0)>v(r_0)$, $u(r_0)>0$ and $u_r(r_0)>0$.  Then $u(r)$ or $v(r)$  blows up in finite time.
\end{itemize}
\end{lemma}

\begin{proof}
(1)  Let  $[0,T)$ be the  maximal interval of existence for $(u,v)$.\\
 {\it Step 1. } {\it Let $I=(r_0,r_1)$ be the interval such that $u(r)>v(r)$ on $I$. We claim that $u_r(r)>0$ on $[r_0,r_1]$}\par
Suppose it is not true,  so there exits  $r_2\in[r_0,r_1]$  such that
$$u_r(r_2)=0\mbox{    and     }  u_r(r)>0\mbox{ on    }[r_0,r_2].$$
Hence,
\begin{equation}\label{eq015}
0\geq u_{rr}(r_2)=(f_2-2f_1)(r_2)
\end{equation}
From the assumptions $u(r_0)=v(r_0)>0$,  $u_r(r)>0$ on $[r_0,r_2)$ and $u(r)>v(r)$ on $(r_0,r_1)$,  we have
$$u(r_2)>0\mbox{ and } u(r_2)\geq v(r_2).  $$
It follows that
\begin{equation}
(f_2-f_1)(r_2)=g(v(r_2))-g(u(r_2))\geq 0
\end{equation}
and
\begin{equation}
\begin{split}
-f_1(r_2)&=-e^{u(r_2)}+2e^{2u(r_2)}-e^{u(r_2)+v(r_2)}\\
&=e^{u(r_2)}(e^{u(r_2)}-1)+e^{u(r_2)}(e^{u(r_2)}-e^{v(r_2)})>0,
\end{split}
\end{equation}
which contradict to \eqref{eq015}. Hence, $u_r(r)>0$ on $[r_0,r_1]$.\\
{\it Step 2. } {\it  We show that $u(r)>v(r)$ on $(r_0,T)$.}\par
Suppose  for the sake of contradiction that there exits $r_3\in(r_0,T)$ such that
$$ u(r_3)=v(r_3)\mbox{   and    } u(r)>v(r)\mbox{   on  }(r_0,r_3).$$
Since $u_r(r)>0$ on  $(r_0,r_3)$ and $u(r_0)>0$, we have
$$u(r)>0\mbox{ on }[r_0,r_3]\mbox{    and   } u(r)>v(r)\mbox{ on    }(r_0,r_3).     $$
Thus,
$$(f_2-f_1)(r)=g(v(r))-g(u(r))>0 \mbox{   for  }r\in(r_0,r_3).$$
By this  and  $(u-v)_r(r_0)>0$, we get
$$r(u-v)_r(r)=r_0(u-v)_r(r)+3\int_{r_0}^r s  (f_2-f_1)(s)ds>0 \mbox{   for   }r\in[r_0,r_3].$$
Obviously, it is a contradiction.\\
{\it Step 3. }{\it We show that $T<\infty$.}\par
We introduce the change of variable
$$t=\ln r,\,\,t_0=\ln r_0 \mbox{   and   }T_0=\ln T$$
The first equation of \eqref{QWE} becomes
\begin{equation}
u''=e^{2t}(f_2-2f_1),\,\,\,    -\infty<t<T_0
\end{equation}
where   $'$ is denoted the differentiation with respect to $t$. Set  $ \widetilde{u}=u+2t$.
On $(t_0,T)$, by applying  $u(r)>v(r)$ on $(r_0,T)$ and $u(r)>0$  on $(r_0,T)$, we have
\begin{equation}\label{eq031}
 \widetilde{u}''> e^{ \widetilde{u}}(e^{ u}-1)\geq \delta e^{ \widetilde{u}},
\end{equation}
for  $\delta\in(0,    e^{u(r_0)}-1    )$. We further restrict $$\delta\in(0, \min\{ e^{u(r_0)}-1 , \frac{1}{2} \widetilde{u}'^2(t_0) e^{- \widetilde{u}(t_0)}      \}     ).$$
Multiplying \eqref{eq031} by $\widetilde{u}'$ and integrating on $(t_0,t)$, we have
$$
\frac{1}{2}(\widetilde{u}'^2(t)-\widetilde{u}'^2(t_0))\geq \delta (e^{ \widetilde{u}(t)}-e^{ \widetilde{u}(t_0)}).
$$
It follows that
$$
\widetilde{u}'(t)\geq \sqrt{2\delta}\,(\frac{1}{8}\widetilde{u}^2(t)+1)
$$
because $\widetilde{u}'(t)>0$ on $(t_0,T)$ and $\widetilde{u}(t)$ on $(t_0,T)$.
It  leads $\widetilde{u}$ to blow up in finite time and thus  $T$ is  finite.\\
(2) Since we do not assume $(u-v)_r(r_0)>0$, $u(r)$ may intersect $v(r)$ on $(r_0,T)$. Thus we consider the following possible cases:
\begin{itemize}
\item[(1)] $u(r)>v(r)$ on $(r_0,T)$.
\item[(2)] $u(r)$ intersects $v(r)$ on $(r_0,T)$.
\end{itemize}
{\it Step 1.  }{\it For the first case, we show $T<\infty$.}\par
Repeating the arguments in  the step 1 of this lemma (1), we find
$$u_r(r)>0\mbox{  on    } (r_0,T),   $$
because we suppose $u(r)>v(r)$ on $(r_0,T)$.
As in  the step 3 of this lemma (1), we conclude that $T<\infty$.\\
{\it Step 2. }{\it For the second case, let $r_1$ be the first intersection point of $u$ and $v$ on $(r_0, T)$.  We show that $T<\infty$. }\par
As in the step 1  of this lemma (1) again, we know
that
$$u_r(r)>0\mbox{  on    } [r_0,r_1].$$
Thus, we have
$$v(r_1)=u(r_1)>0,\,\, v_r(r_1)>0\mbox{    and    } (v-u)_r(r_1)>0$$
and the second case follows from this lemma (1).
\end{proof}
\par
{\bf Proof of Theorem \ref{thm2} }\\
{\it Step 1. }  For $r>0$, we denote $$F_U(r)=\max\{u(r),v(r)\}$$ and
$$
G_U(r)=
\left\{
\begin{array}{lll}
\max\{u_r(r),v_r(r)\}&\mbox{  if  }&u(r)=v(r)\\
\frac{d}{dr}F_U(r)&\mbox{  if  }&u(r)\not=v(r)
\end{array}
\right.
$$
{\it Step 2.}   Suppose $F_U(r)$ attains its positive maximum value at some point $r_M\in[0,\infty)$. By symmetry, we may assume that $u(r_M)\geq v(r_M)$.
Thus we obtain
\begin{equation}\label{qwe005}
0\geq f_2(r_M)-2f_1(r_M)=g(v(r_M))-g(u(r_M))-f_1(r_M).
\end{equation}
Note that
$$g(v(r_M))-g(u(r_M))\geq 0\mbox{   if   }u(r_M)>0\mbox{   and   }u(r_M)\geq v(r_M),$$
and
$-f_1(r_M)=e^{u(r_M)}[(e^{u(r_M)}-1)+(e^{u(r_M)}-e^{v(r_M)}    ]>0$   if    $u(r_M)>0$  and $u(r_M)\geq v(r_M).$
It is a contradiction to \eqref{qwe005}. We conclude that either $F_U(r)$ attains non-positive maximum values on $[0,\infty)$
or $F_U(r)$ never attains a maximum on on $[0,\infty)$.\\
{\it Step 3. } {\it Suppose $F_U(r)$ never attains a maximum on on $[0,\infty)$, we show that $F_U(r)\leq 0$.}\par
Suppose otherwise that $F_U>0$ at some point $r_0\in(0,\infty)$. Since we assume that $(u,v)$ is an entire solution, by Lemma \ref{blowup},   $G_U(r_0)\leq 0$.  In view of the boundary conditions at $0$ and the step 2, it is
a contradiction that $F_U$ attains positive maximum value on $[0,r_0]$. Hence, $F_U(r)\leq 0$ on $(0,\infty)$ if it never attains a maximum on on $[0,\infty)$.\\
{\it Step 4.} It is easy to see that $F_U(r_0)=0$ if only if   $u(r_0)=v(r_0)=0$. By this and step 3,
we have
$$u_r(r_0)=v_r(r_0),$$
and thus $u(r)=v(r)=0$ for $r\in(0,\infty)$. We  conclude that
$$u(r),v(r)<0\mbox{    on   }(0,\infty),$$
unless $u=v\equiv 0$.

\section{APRIORI ESTIMATE  }
In this section, we present some  apriori estimates of  the behavior of $(u(r),v(r))$. The most important result is to prove $\max(u(r),v(r))<\log\frac{1}{2}$ under some condition. See Lemma \ref{co1} and Lemma \ref{ll}. This estimate is a crucial step toward proving the positivity of $f_i$, $i=1,2$.       We first have the following simple observation.

\begin{lemma}\label{remark1}
Let $(u,v)$ be the solution of \eqref{QAZ}. Suppose that $\lim_{r\to\infty}e^{u(r)}$ and   $\lim_{r\to\infty}e^{v(r)}$ exists. Then $\lim_{r\to\infty}(e^{u(r)},e^{v(r)})$ must be one of the following:
$$\begin{array}{ccc}
(a)\,(1,1).& \hspace{3cm}  &
(b)\,(\frac{1}{2},0).\\
(c)\,(0,\frac{1}{2}).&  \hspace{3cm}  &
(d)\,  (0,0).
\end{array}
$$
\end{lemma}

If the derivative of one of $u(r)$ and $v(r)$ must be negative   on an interval $I$, then  $u(r)$ and $v(r)$ cannot increase simultaneously on  $I$ and  we say that $u(r)$ and $v(r)$ satisfy {\it non-simultaneous increasing condition} (for brevity, { nsi}-condition) on $I$. If the derivative of one of $u(r)$ and $v(r)$ must be positive   on an interval $I$, then  $u(r)$ and $v(r)$ cannot decrease simultaneously on  $I$ and  we say  that $u(r)$ and $v(r)$ satisfy {\it non-simultaneous decreasing condition} (for brevity, { nsd}-condition) on $I$.

\begin{definition}
\begin{enumerate}
\item[(1).] We say that a function $f(r)$ has an $S_{[a,b]}$-profile if
$$f_r(a)=f_r(b)=0\,\,\text{  and  }f_r(r)\geq 0\text{ on }(a,b),$$
 and a function $f(r)$ has a reversive $S_{[a,b]}$-profile if
$$f_r(a)=f_r(b)=0\,\,\text{  and  }f_r(r)\leq 0\text{ on }(a,b).$$
\item[(2).] $f(r)$ has an $S$-profile on $[c,d]$ if $f(r)$ has  an $S_{[a,b]}$-profile, where  $[a,b]\subset[c,d]$. Similarly, $f(r)$ has a reversive  $S$-profile on $[c,d]$ if $f(r)$ has  a reversive  $S_{[a,b]}$-profile, where   $[a,b]\subset[c,d]$.
\end{enumerate}
\end{definition}

The following lemmas will be used to prove Theorem \ref{thm1}.

\begin{lemma}\label{1}
Assume that $u(r_0)>v(r_0)$ and $u_r(r_0)\geq 0 \geq v_r(r_0)$. Then $f_1(r_0)\geq 0$.
\end{lemma}
\begin{proof}
Note that $f_1(r_0)=e^{u(r_0)}-2e^{2u(r_0)}+e^{(u+v)(r_0)}$ is positive whenever $v(r_0)<u(r_0)\leq\log\frac{1}{2}$. We only need to consider the case that $u(r_0)>\log\frac{1}{2}$.
Suppose for the sake of contradiction that  $f_1(r_0)<0$.\\
{\it Step 1.} {\it We   show that $u_r(r)>0$ and $v_r(r)<0$ on some interval  $(r_0,r_1)  \subseteq(r_0,\infty)$.}\par
Since we suppose $u(r_0)>v(r_0)$, it is obvious that
\begin{equation}
f_2(r_0)=e^{v(r_0)}-2e^{2v(r_0)}+e^{u(r_0)+v(r_0)}>0.
\end{equation}
By this and $f_1(r_0)<0$, we have
$$(f_2-2f_1)(r_0)>0\mbox{  and  }(f_1-2f_2)(r_0)<0.$$
It follows that
\begin{equation}\label{eq001}
(f_2-2f_1)(r)>0\mbox{  and  }(f_1-2f_2)(r)<0\mbox{   on  some interval }[r_0,r_1).
\end{equation}
Hence,
\begin{equation}\label{eq013}
ru_r(r)=r_0u_r(r_0)+\int_{r_0}^r s(f_2-2f_1)ds>0\mbox{  for }r\in(r_0, r_1]
\end{equation}
and
\begin{equation}\label{eq014}
rv_r(r)=r_0v_r(r_0)+\int_{r_0}^r s(f_1-2f_2)ds<0\mbox{  for }r\in(r_0, r_1].
\end{equation}
\\
{\it Step 2.} {\it  Let $r_2=\sup\{s\,|\,s>r_0 \mbox{   and   } u_r>0\mbox{   on  } (r_0,s)  \}.$  We claim that $v_r(r)<0$ for $r\in(r_0,r_2)$.}\par
Suppose  for the sake of contradiction that there exists $r_3\in(r_0,r_2)$, such that
 $v_r(r_3)=0$  and $v_r(r)<0$ for $r\in(r_0,r_3)$. It follows that
\begin{equation}\label{eq071}
0\leq v_{rr}(r_3)=(f_1-2f_2)(r_3),
\end{equation}
which implies
\begin{equation}\label{eq003}
f_1(r_3)\geq 2f_2(r_3)>0.
\end{equation}
Inequality  \eqref{eq003} follows from  $u>v$  on $(r_0,r_3]$. Since we suppose $f_1(r_0)<0$ and by \eqref{eq003}, there exists $r_4\in (r_0,r_3)$ such that
$$
f_1(r_4)=0\mbox{   and   }f_1(r)>0\mbox{  for } r\in(r_4,r_3).
$$
It follows that $\frac{d}{dr}f_1(r_4)\geq 0$. On the other hand,
the calculation of  $\frac{d}{dr}f_1(r_4)$ gives
\begin{equation}\label{ee1}
\begin{split}
\frac{d}{dr}f_1(r_4)&=(e^u-4e^{2u}+e^{u+v})|_{r_4}u_r(r_4)+e^{u+v}|_{r_4}v_r(r_4)\\
&=-2e^{2u}|_{r_4}u_r(r_4)+e^{u+v}|_{r_4}v_r(r_4)
\end{split}
\end{equation}
where  $f_1(r_4)=0$ is used. Since $u_r(r_4)>0$ and $v_r(r_4)<0$, we have that the right hand side of \eqref{ee1} is negative which is a contradiction.
So, $v_r(r)<0$ for $r\in(r_0, r_2)$ is proved.\\
\\
{\it Step 3.} {\it We now show that $r_2=\infty$.}\par If it is not true, then
$$
u_r(r_2)=0\mbox{  and  }u_r(r)>0\mbox{  for  }r\in(r_0,r_2).
$$
It follows that
\begin{equation}
0\geq u_{rr}(r_2)=(f_2-2f_1)(r_2).
\end{equation}
By the step 2, we have  $u(r_2)>v(r_2)$ and thus
$$2f_1(r_2)\geq f_2(r_2)>0.$$
By this and $f_1(r_0)< 0$, there exits $r_5\in(r_0,r_2)$, such that
$$
f_1(r_5)=0\mbox{   and   }f_1(r)>0\mbox{  for } r\in(r_5,r_2),
$$
which implies $\frac{d}{dr}f_1(r_5)\geq 0$.
By computing $\frac{d}{dr}f_1(r_5)$ as in \eqref{ee1}, we have
$$\frac{d}{dr}f_1(r_5)=-2e^{2u}\big|_{r_5}u_r(r_5)+ e^{u+v}\big|_{r_5}v_r(r_5)<0,   $$
a contradiction. We conclude that $r_2=\infty$.\\
\\
{\it Step 4.}  Since $u_r>0 >v_r$ on $(r_0,\infty)$, we know that
$\lim_{r\to\infty}e^u$ and  $\lim_{r\to\infty}e^v$ exist. By the assumption that $u(r_0)>v(r_0)$ and $u(r_0)>\log\frac{1}{2}$,
$$\lim_{r\to\infty}e^{u(r)}>\frac{1}{2} \mbox{  and  } \lim_{r\to\infty}e^{v(r)}< e^{u(r_0)}<1, $$
which is a contradiction to Lemma \ref{remark1}. Hence, if  $u(r_0)>v(r_0)$ and $u_r(r_0)\geq 0 \geq v_r(r_0)$, then $f_1(r_0)>0$.
\end{proof}

\begin{remark}
When $(N_1,N_2)=(0,0)$, $u_r(0)=v_r(0)=0$.    Lemma \ref{1} suggests that $\Omega\not= \mathbb{R}^-\times \mathbb{R}^-$.
\end{remark}

Note that the definition of the $S_{[a,b]}$-profile (resp.  reversive $S_{[a,b]}$-profile)  of $f(r)$ does not require $f(r)$ attains local maximum at $b$ (resp.  $a$). However, we can extend the interval $[a,b]$ until $f(r)$ attains local maximum at $c\in(b,\infty)$(resp. $c\in[0,a)$).  Hence, the following two lemmas shows, under certain conditions, the local  maximum value of the upper function is less than $\log\frac{1}{2}$, which is crucial to the proof of Theorem \ref{thm1}.

\begin{lemma}\label{co1}
\begin{itemize}
\item[(1).] Assume that $u(r)$ has an $S_{[r_0,r_1]}$-profile, $v(r)<u(r)$ in $(r_0,r_1)$ and $v_r(r_0)< 0$. Then $u(r_1)<\log\frac{1}{2}$.
\item[(2).] Assume that $u(r)$ has a reversive $S_{[r_0,r_1]}$-profile, $v(r)<u(r)$ on $(r_0,r_1)$  and $v_r(r_1)> 0$. Then $u(r_0)<\log\frac{1}{2}$.
\end{itemize}
\end{lemma}
\begin{proof}
(1) We prove this lemma by contradiction. Assume that $u(r_1)\geq\log\frac{1}{2}$.\\
\\
{\it Step 1.}{\it  We first show that  $v_r(r)<0$ on $[r_0,r_1]$. }\par
Since we suppose $v(r)<u(r)$ on $(r_0,r_1)$ and $u_r(r_0)=0> v_r(r_0)$, we have
$$r(u+2v)_r(r)=r_0(u+2v)_r(r_0)-3\int_{r_0}^r sf_2ds<0\,\mbox{    for } r\in[r_0,r_1].$$
By this and $u_r(r)\geq 0$ on $[r_0,r_1]$,  $v_r(r)<0$ on $[r_0,r_1]$.\\
\\
{\it Step 2.} {\it We  show that there exists  $r_2\in[r_0, r_1)$ such that
$$(f_2-2f_1)(r)<0\mbox{   on   } (r_2,r_1) \mbox{   and  }(f_2-2f_1)(r_2)=0 . $$}\par

Since $u_r(r)\geq 0$ on $[r_0,r_1]$ and $u_r(r_0)=u_r(r_1)=0$, then
\begin{equation}
0\leq u_{rr}(r_0)=(f_2-2f_1)(r_0)
\end{equation}
and
\begin{equation}
0\geq u_{rr}(r_1)=(f_2-2f_1)(r_1).
\end{equation}
Hence, if $(f_2-2f_1)(r_1)<0$, we are done.  If $(f_2-2f_1)(r_1)=0$, by computing $\frac{d}{dr}(f_2-2f_1)(r_1)$, we have
\begin{equation}
\begin{split}
\frac{d}{dr}(f_2-2f_1)(r_1)&=2(e^{u}-2e^{2u}-e^{2v})\big|_{r_1}v_r(r_1)-2(e^u-4e^{2u}+\frac{1}{2} e^{u+v})\big|_{r_1}u_r(r_1)\\
&=2(e^{u}-2e^{2u}-e^{2v})\big|_{r_1}v_r(r_1)>0,
\end{split}
\end{equation}
where $u_r(r_1)=0$ and $u(r_1)\geq\log\frac{1}{2}$ are used. Thus, there exists $r_2$ with
$$(f_2-2f_1)(r)<0\mbox{   on   } (r_2,r_1) \mbox{   and  }(f_2-2f_1)(r_2)=0. $$
Consequently,
\begin{equation}\label{eq008}
\frac{d}{dr}(f_2-2f_1)(r_2)\leq 0.
\end{equation}\\
\\
{\it Step 3.} {\it The calculation of $\frac{d}{dr}(f_2-2f_1)(r_2)$ will lead to a contradiction.}\par
If $e^{u(r_2)}\geq \frac{1}{2}$, by using $(f_2-2f_1)(r_2)=0$ and $v_r(r_2)<0\leq u_r(r_2)$,  we have
\begin{equation}\label{eq026}
\begin{split}
&\frac{d}{dr}(f_2-2f_1)(r_2)\\
=&2(e^{u}-2e^{2u}-e^{2v})\big|_{r_2}v_r(r_2)-2(e^u-4e^{2u}+\frac{1}{2} e^{u+v})\big|_{r_2}u_r(r_2)>0.
\end{split}
\end{equation}
It contradicts to \eqref{eq008}.\par

If $e^{u(r_2)}<\frac{1}{2}$,  there is $r_3\in(r_2,r_1]$ such that
$e^{u(r_3)}=\frac{1}{2}$. At $r=r_3$, one has that
$$(f_2-2f_1)(r_3)=\frac{1}{2}e^{v(r_3)}-2e^{2v(r_3)}\leq0,$$
which implies that $e^{v(r_3)}\geq \frac{1}{4}$. It follows  that
$$e^{v(r_2)}>e^{v(r_3)}\geq\frac{1}{4}.$$
and thus $e^{v(r_2)}-2e^{2v(r_2)}=g(v(r_2))<\frac{1}{8}$.
By this and $(f_2-2f_1)(r_2)=0$, we have
\begin{equation}\label{eq009}
\begin{split}
\frac{d}{dr}(f_2-2f_1)(r_2)&=(e^v-4e^{2v}-e^{u+v})\big|_{r_2}v_r(r_2)-2(e^v-2e^{2v}-2 e^{2u})\big|_{r_2}u_r(r_2)\\
&\geq-e^{(u+v)(r_2)}v_r(r_2) >0,
\end{split}
\end{equation}
a contradiction to \eqref{eq008}.   Hence, $u(r_1)<\log\frac{1}{2}$.\\
\\
(2) Heuristically, this part can be viewed as the reflection of the part (1).  Although equations \eqref{QWE} change after reflection, we can  still apply the techniques of the part (1) to prove part (2). Hence, we omit the details of the proof and only sketch it.
This proof can be based on the following three steps:
\begin{itemize}
\item[\it Step 1.] We first show that  $v_r(r)>0$ on $[r_0,r_1]$.
\item[\it Step 2.]  We   show   there exists  $r_2\in(r_0, r_1]$ such that
$$f_2-2f_1<0\mbox{   on   } (r_0,r_2) \mbox{   and  }(f_2-2f_1)(r_2)=0.$$
\item[\it Step 3.]  The calculation of $\frac{d}{dr}(f_2-2f_1)(r_2)$ will lead to a contradiction.
\end{itemize}

\end{proof}

%
%

\begin{lemma}\label{ll}\begin{itemize}
\item[(1).] Assume that $u(r)$ has a $S_{[r_0,r_2]}$-profile, $u(r_1)=v(r_1)<\log\frac{1}{2}$ for some $r_1\in(r_0,r_2)$ and $v(r)$ is decreasing on $(r_0,r_1)$.     Then $u(r_2)<\log\frac{1}{2}$.
\item[(2).] Assume that $u(r)$ has a reversive $S_{[r_0,r_2]}$-profile,  $u(r_1)=v(r_1)<\log\frac{1}{2}$  for some $r_1\in(r_0,r_2)$ and $v(r)$ is increasing on $(r_0,r_1)$. Then $u(r_0)<\log\frac{1}{2}$.
\end{itemize}
\end{lemma}
\begin{proof} We only prove the first part. As in Lemma \ref{co1}, the second part can be viewed as the reflection version of the first part heuristically.  We prove this lemma by contradiction.  Suppose that  $u(r_2)\geq\log\frac{1}{2}$.\\
{\it Step 1.} {\it We first show that $(u+v)_r(r)<0$ on $(r_0,r_2)$.}\par
 Since $u_r(r_0)=0$ and $u_r(r)\geq 0$ on $[r_0,r_2]$, we have $$0\leq u_{rr}(r_0)=(f_2-2f_1)(r_0)$$ and thus $f_2(r_0)>0$.   Since $u(r)$ is increasing and $v(r)$ is decreasing  on $(r_0, r_2)$, it follows that $(1-2e^{v(r)}+e^{u(r)})$ is increasing   on $(r_0, r_1)$.
Thus, we  have
$$f_2(r)=e^{v(r)}(1-2e^{v(r)}+e^{u(r)})>0\,\, \mbox{  on  }\,\,(r_0,r_1].$$
It follows that
\begin{equation}\label{e222}
r(v_r+u_r)(r)=r_0(v_r+u_r)(r_0)-\int_{r_0}^rs(f_1+f_2)ds<0,\text{      }r\in(r_0, r_1],
\end{equation}
because $v_r(r_0)\leq 0=u_r(r_0)$ and $f_1(r)>0$ for $r\in(r_0,r_1)$.\par
Note that $v(r)<u(r)$  on $(r_1,r_2]$ and  $u_r(r)\geq 0\geq v_r(r)$ on $[r_0,r_2]$. By applying Lemma \ref{1}, for each $r\in(r_1,r_2)$, we have $f_1(r)\geq0$ for $r\in(r_1,r_2)$. By \eqref{e222} again, we get
\begin{equation}
(u+v)_r(r)<0\mbox{   for  }r\in(r_0,r_2],
\end{equation}
which also implies $v_r(r)<0$ on $(r_0,r_2]$.\\
\\
{\it Step 2.} \textit{ We claim that $(f_2-2f_1)(r)<0$  on $[r_1,r_2)$.}\par
Note that
\begin{equation}\label{eq004}
\begin{split}
(f_2-2f_1)(r_1)&=-2e^{u(r_1)}+4e^{2u(r_1)}-e^{u(r_1)+v(r_1)}+e^{v(r_1)}-2e^{2v(r_1)}\\
&=-e^{u(r_1)}+e^{2u(r_1)}<0
\end{split}
\end{equation}
where $u(r_1)=v(r_1)$ is used.
As in the step 2 of Lemma \ref{co1} (1), we know $f_2-2f_1<0$ on some interval $(s,r_2)\subset[r_1,r_2)$.
Consequently, if the claim is not true, then there exists $r_3\in(r_1,r_2)$, such that
$$(f_2-2f_1)(r)<0 \mbox{   for   }r\in(r_3, r_2)\mbox{  and   } (f_2-2f_1)(r_3)=0,   $$
which implies \begin{equation}\label{eq010}
\frac{d}{dr}(f_2-2f_1)(r_3)\leq 0.
\end{equation}\par
As in the step 3 of Lemma \ref{co1} (1), the calculation of $\frac{d}{dr}(f_2-2f_1)(r_3)$ will lead to a contradiction.
 Note that $u_r(r_3)\geq 0>v_r(r_3)$.  If $e^{u(r_3)}\geq \frac{1}{2}$,  as in \eqref{eq026},
$$\frac{d}{dr}(f_2-2f_1)(r_3)>0,$$
a contradiction to \eqref{eq010}.  If $e^{u(r_3)}<\frac{1}{2}$, then there is $r_4\in(r_3,r_2]$ such that
$e^{u(r_4)}=\frac{1}{2}$. At $r=r_4$, one has that
$$(f_2-2f_1)(r_4)=\frac{1}{2}e^{v(r_4)}-2e^{2v(r_4)}\leq0$$
which implies that $e^{v(r_4)}\geq \frac{1}{4}$. It follows  that
$$e^{v(r_3)}>e^{v(r_4)}\geq\frac{1}{4}.$$  Hence, as in \eqref{eq009},
$$\frac{d}{dr}(f_2-2f_1)(r_3)>0,$$
a contradiction to \eqref{eq010}.  We conclude here that $(f_2-2f_1)(r)<0$  for  $r\in[r_1,r_2)$.\\
\\
{\it Step 3. } {\it  We show that $u(r_1)=v(r_1)>\log\frac{1}{4}$, which is important in the next step. }\par
Since  we suppose  $e^{u(r_2)}\geq \frac{1}{2}$ and $u(r_1)=v(r_1)<\log\frac{1}{2}$, there exits $r_5\in(r_1, r_2]$ such that $e^{u(r_5)}=\frac{1}{2}$. By the step 2,  $f_2-2f_1<0$ on $(r_1,r_2)$,
we have  $$(f_2-2f_1)(r_5)\leq 0.$$    Hence,
$$\frac{1}{2}e^{v(r_5)}-2e^{2v(r_5)}=(f_2-2f_1)(r_5)\leq 0,  $$
which implies
\begin{equation}\label{eq024}
e^{v(r_5)}\geq \frac{1}{4}.
\end{equation}
Since $v_r<0$ on $(r_0,r_2]$ and $r_1<r_5$, then
\begin{equation}\label{eq042}
u(r_1)=v(r_1)>v(r_5)\geq\log\frac{1}{4}.
\end{equation}
\\
{\it Step 4.} {\it We claim that there exits $r_6\in(r_0, r_1)$ such that
$$(f_1-f_2)(r)>0\mbox{    for  }r\in(r_6,r_1).$$}\par
Note that $(f_1-f_2)(r_1)=0$. To prove this claim, it suffices to show  $\frac{d}{dr}(f_1-f_2)(r_1)<0$.
We compute $\frac{d}{dr}(f_1-f_2)(r_1)$.
\begin{equation}\label{eq011}
\begin{split}
\frac{d}{dr}(f_1-f_2)(r_1)&=\frac{d}{dr}\big[(e^u-2e^{2u})-  (e^v-2e^{2v}) \big]\big|_{r_1} \\
&=(e^{u(r_1)}-4e^{2u(r_1)})(u_r(r_1)-v_r(r_1)),
\end{split}
\end{equation}
where $u(r_1)=v(r_1)$ is used.
By \eqref{eq042} and $u_r\geq 0>v_r$ on $(r_0,r_2)$,
$$\frac{d}{dr}(f_1-f_2)(r_1)=(e^{u(r_1)}-4e^{2u(r_1)})(u_r(r_1)-v_r(r_1))<0.$$ \\
{\it Step 5.} Recall that
$$u_r(r_0)=0\text{   and   } (f_2-2f_1)(r_0)=u_{rr}(r_0)\geq 0,$$
which implies $(f_1-f_2)(r_0)\leq -f_1(r_0)<0$.
Combining this and  the step 4, we know  that
there is $r_7\in(r_0,r_1)$ such that
$$(f_1-f_2)(r_7)=0.$$
Note that $u(r_7)<v(r_7)$ and $g(u(r_7))-g(v(r_7))=f_1(r_7)-f_2(r_7)=0$.
We obtain that
\begin{equation}\label{e4}
u(r_7)<\log\frac{1}{4}<v(r_7)<\log\frac{1}{2}.
\end{equation}
Note that $(u+v)_r<0$ on $(r_0,r_2)$, and $r_7<r_5$.
By this, \eqref{e4}  and \eqref{eq024},  we obtain
\begin{equation}
\log\frac{1}{4}+\log\frac{1}{2}\leq (u+v)(r_5)<(u+v)(r_7)<  \log\frac{1}{2}+\log\frac{1}{4},
\end{equation}
a contradiction. Hence,  $u(r_2)<\log\frac{1}{2}$.\\
\end{proof}

\section{ASYMPTOTIC BEHAVIOR (1): WITHOUT INTERSECTION}
In this section, under certain conditions, we discuss the asymptotic behaviors of $u$  and $v$ when they do have intersection for $r$ sufficiently large.  We first exclude one special case:
\begin{itemize}
\item[$(*)$]  There exists $R_0>0$, so that
one of  $u(r)$ and $v(r)$ is decreasing to $\log\frac{1}{2}$, and the another one is decreasing to $-\infty$ for $r\geq R_0$.
\end{itemize}
\begin{lemma}\label{lemma84}
There is no solution of \eqref{QWE}  which satisfies condition $(*)$.
\end{lemma}
\begin{proof}
{\it Step 1. } With out loss of generality, we assume that $u$ and $v$ are decreasing to $\log\frac{1}{2}$ and $-\infty$  on $(R_0,\infty)$ respectively. We write $u=\log\frac{1}{2}+\hat{u}$ (Here, $\hat{u}>0$ on $(R_0,\infty)$).   Thus, there exists $R_1>>R_0$ so that
\begin{equation}\label{e00011}
\Delta \hat{u}\geq  \frac{1}{4}e^{v}+\hat{u}.
\end{equation}
We will show there is no such solution $u=\log\frac{1}{2}+\hat{u}$ which satisfies \eqref{e00011}.\\
{\it Step 2.} {\it We show that $f_2\in L^1(\mathbb{R}^2)$ which is important to the  estimates of $\lim_{r\to\infty}ru_r(r)$ and $\lim_{r\to\infty}rv_r(r)$.}\par
We may assume that $v(r)<\log\frac{1}{4}$ on $(R_1,\infty)$.  Thus, we only need to show that $\int_{R_1}^{\infty}sf_2(s)ds<\infty$.
For $r>>R_1$,
\begin{equation}\label{e010}
\begin{split}
rv_r(r)=&2N_2+\int_0^{r}s(f_1-2f_2)(s)ds\\
=&2N_2+\int_0^{R_1}s(f_1-2f_2)(s)ds+\int_{R_1}^{r}s(f_1-2f_2)(s)ds\\
\leq & 2N_2+\int_0^{R_1}s(f_1-2f_2)(s)ds-\int_{R_1}^rsf_2(s)ds
\end{split}
\end{equation}
where $(f_1-f_2)(s)=g(u(s))-g(v(s))<0$ on $(R_1,\infty)$ is used.
Suppose that $\int_{R_1}^{\infty}sf_2(s)ds=\infty$, then
$$rv_r(r)\to-\infty\mbox{   as  }r\to\infty. $$
It implies that $\int_{R_1}^{\infty}se^{v(s)}ds<\infty$ and thus
$\int_{R_1}^{\infty}sf_2(s)ds<\infty$ where $u<0$ is used.
We conclude that $f_2\in L^1(\mathbb{R}^2)$.\\
{\it Step 3.} {\it We show that $\lim_{r\to\infty}rv_r(r)<-2$.}\par
One can see that $rv_r(r)$ is a decreasing function on  $(R_1,\infty)$ from \eqref{e010}. Hence, if $\lim_{r\to\infty}rv_r(r)\geq- 2$, then
$$rv_r(r)\geq -2\mbox{   on  }(R_1,\infty),$$
which makes the $L^1(\mathbb{R}^2)$-integrability of $f_2$ fail.
Therefore, $\lim_{r\to\infty}rv(r)<-2$.\\
{\it Step 4.} {\it We show that $f_1\in L^1(\mathbb{R}^2)$.}\par
We split the nonlinear term $f_2-2f_1$ into
\begin{equation}\label{e00003}
\Big(e^v-2e^{2v}-e^{u+v}\Big)+ 2\Big(e^{u}-2e^{2u}\Big).
\end{equation}
Note that the first term of \eqref{e00003} is in $L^1(\mathbb{R}^2)$ by the step 2.,  and the second term of \eqref{e00003}  is negative for $r>R_1$. Hence,
$$\int_{R_1}^{\infty} s\Big(e^{u(s)}-2e^{2u(s)}\Big)ds >-\infty.$$
where $$ru_r(r)=R_1u_r(R_1)+\int_{R_1}^{\infty} s(e^{v(s)}-2e^{2v(s)}-e^{(u+v)(s)})ds+2 \int_{R_1}^{\infty} s(e^{u(s)}-2e^{2u(s)})ds          $$
and  the existence of  $\lim_{r\to \infty}u(r)$ are used.
It follows that $f_1\in L^1(\mathbb{R}^2)$.\\
{\it Step 5.} By the steps 2, 3 and 4, we have $\lim_{r\to\infty}ru_r(r)=0$ and $\lim_{r\to\infty}rv_r(r)=-\beta$ for some constant $\beta>2$.
Thus, for $R_2>R_1$ large enough, we have
\begin{equation}\label{e0001}
(r\hat{u}_r(r))_r\geq  \frac{1}{4}r^{1-\beta}+r\hat{u}(r)\mbox{   for    } r>R_2.
\end{equation}
Since $\hat{u}>0$ on $(R_2,\infty)$,
\begin{equation}\label{e0002}
(r\hat{u}_r(r))_r\geq  \frac{1}{4}r^{1-\beta}\mbox{   for    }  r>R_2.
\end{equation}
Integrating \eqref{e0002} from $r>R_2$ to $\infty$, we obtain
\begin{equation}\label{e0003}
-r\hat{u}_r(r)\geq  \frac{1}{4(\beta-2)}r^{2-\beta},
\end{equation}
where $\lim_{r\to\infty}ru_r(r)=0$ and $\beta>2$ are used.  Dividing \eqref{e0003} by $r$ and integrating it from  $r>R_2$ to $\infty$, we have
\begin{equation}\label{e0004}
\hat{u}(r)\geq  \frac{1}{4(\beta-2)^2}r^{2-\beta},
\end{equation}
where $\lim_{r\to\infty}u(r)=0$ and $\beta>2$ are used.\par
Plugging \eqref{e0004} into \eqref{e0001}, we have
\begin{equation}\label{e0005}
(r\hat{u}_r(r))_r\geq  \frac{1}{4}r^{1-\beta}  +\frac{1}{4(\beta-2)^2}r^{3-\beta}>\frac{1}{4(\beta-2)^2}r^{3-\beta}\mbox{   for    }  r>R_2.
\end{equation}
If $\beta>4$, integrating \eqref{e0005}  as in \eqref{e0002} and \eqref{e0003}, we obtain
$$\hat{u}(r)\geq \frac{1}{4(2-\beta)^2(4-\beta)^2}r^{4-\beta}.$$
If $\beta\leq 4$, we obtain a contradiction when we do integration as in  \eqref{e0002} and \eqref{e0003}.
Since $\beta$ is finite,  one can repeat this argument  at most a finite number of times to get a contradiction.

\end{proof}

In the first part of the  following lemma, we show that if
\begin{equation*}
 u(r)>v(r)\mbox{   on    }(r_0,\infty),\,\, u_r(r_0)\geq v_r(r_0),\mbox{   and      }(u+2v)_r(r_0)\leq 0,
\end{equation*}
then the behavior of $(u,v)$ must be case (1) described in the statement of Theorem \ref{thm1}.
In the second part of the  following lemma, we show that if
$$u(r)>v(r)\mbox{   on       }(r_0,\infty)\mbox{   and        }(u+v)_r(r)>0\mbox{ on  }(r_0,\infty),$$
$(u,v)$ must be a topological solution.
\begin{lemma}\label{lemma7}
\begin{itemize}
\item[(1)] Suppose that $u(r)>v(r)$ on $(r_0,\infty)$, $u_r(r_0)\geq v_r(r_0)$ and $(u+2v)_r(r_0)\leq 0$. Then\par
 $u(r)$ and $v(r)$ are less than $\log\frac{1}{2}$ on some interval $(R_1,\infty)\subseteq (r_0,\infty)$.
\item[(2)] Suppose that $u(r)>v(r)$ on $(r_0,\infty)$, and $(u+v)_r(r)>0$ on $(r_0,\infty)$ . Then
$$u(r)\mbox{  and  }v(r)  \mbox{ are increasing to   } 0 \mbox{   as   }r\to\infty.$$
\end{itemize}
\end{lemma}
\begin{proof}
(1) {\it Step 1.} {\it  We show $u(r)$ and $v(r)$ satisfy the nsi-condition on $(r_0,\infty)$.}\par
Since we suppose $u(r)>v(r)$ on $(r_0,\infty)$ and $(u+2v)_r(r_0)\leq 0$, then
\begin{equation}\label{eq030}
r(u+2v)_r(r)=r_0(u+2v)_r(r_0)-3\int_{r_0}^rsf_2(s)ds<0\mbox{   for    }r\in(r_0,\infty).
\end{equation}
Consequently, $u(r)$ and $v(r)$ satisfy the nsi-condition on $(r_0,\infty)$.\\
{\it Step 2.}    We consider the following possible cases:
\begin{itemize}
\item[(a)] $u(r)$ oscillates  on $(r_0,\infty)$.
\item[(b)] $u(r)$ is increasing on some interval $(r_1,\infty)\subseteq(r_0,\infty)$.
\item[(c)] $u(r)$ is decreasing on some interval $(r_2,\infty)\subseteq(r_0,\infty)$.
\end{itemize}
{\it Step 2.1. }  Suppose  $u(r)$ oscillates  on $(r_0,\infty)$, which implies  $u(r)$ has infinitely many S-profile on $(r_0,\infty)$.
 Let $S_{[\alpha,\beta]}$ be its first S-profile on $(r_0,\infty)$.
Combining the nsi-condition on $(r_0,\infty)$ and Lemma \ref{co1} (1), the local maximum values of $u(r)$ on $(\alpha,\infty)$ are less than $\log\frac{1}{2}$. Since we suppose $u$ oscillates  on $(r_0,\infty)$, then
$$u(r)<\log\frac{1}{2}\, \mbox{  for   }r\in(\alpha,\infty).   $$\\
{\it Step 2.2. }  Suppose that $u(r)$ is increasing  on some interval $(r_1,\infty)\subset (r_0,\infty)$. By the nsi-condition on $(r_0,\infty)$,  $v(r)$ is strictly decreasing  on $(r_1,\infty)$. Lemma \ref{remark1} suggests
$$u(r)\mbox{   is increasing to  }\log\frac{1}{2}\mbox{ and }v(r) \mbox{     is decreasing to   }-\infty\mbox{ as }r\to\infty.$$\\
{\it Step 2.3. }  Suppose   $u(r)$ is decreasing on some interval $(r_2,\infty)\subseteq(r_0,\infty)$.  Then either $u(r)$ is decreasing on $(r_0,\infty)$  or there exits $r_3\in (r_2,\infty)$ such that $u$ attains local maximum at $r_3$ and $u$ is decreasing on $(r_3,\infty)$.
 Note that $v_r(r_3)<0$ because of \eqref{eq030}.  Let $r_4=r_0$ if the first case holds, and $r_4=r_3$ if the second case holds.\par

$u(r)$ is decreasing on $(r_4,\infty)$ implies the existence of $\lim_{r\to\infty}e^{u(r)}$. Thus, if  $\lim_{r\to \infty}e^u(r)<\frac{1}{2}$, then $u(r)$ is less than $\log\frac{1}{2}$ on some interval $(r_5,\infty)\subseteq(r_4,\infty)$.\par

We now show that if $\lim_{r\to\infty}e^{u(r)}\geq \frac{1}{2}$, then $v(r)$ is decreasing on  $(r_4,\infty)$.
Since we suppose $u(r)>v(r)$ and $u(r)\geq\log\frac{1}{2}$  for $r\in(r_4,\infty)$, then
$$f_2(r)-f_1(r)=g( v(r))-g(u(r))>0 \mbox{  on  }(r_4, \infty). $$
Hence,
$$r(u-v)_r(r)=r_4(u-v)_r(r_4)+3\int_{r_4}^r s(f_2-f_1)ds>0\mbox{   on  }(r_4,\infty). $$
It follows that $v_r(r)<u_r(r)\leq 0 $ on $(r_4,\infty)$, and thus $\lim_{r\to\infty}e^{u(r)}$ and $\lim_{r\to\infty }e^{v(r)}$ exist. It is a contradiction by Lemma \ref{remark1} and Lemma \ref{lemma84}.\\
(2) {\it Step 1.} As in the proof of (1), we consider  the following possible cases:
\begin{itemize}
\item[(1)] $u(r)$ oscillates  on $(r_0,\infty)$.
\item[(2)] $u(r)$ is increasing on some interval $(r_1,\infty)\subseteq(r_0,\infty)$.
\item[(3)] $u(r)$ is decreasing on some interval $(r_2,\infty)\subseteq(r_0,\infty)$.
\end{itemize}
It will be shown that only the case (2) is possible.\\
{\it Step 2. } {\it We show $u(r)$ cannot oscillate  on $(r_0,\infty)$. }\par
Suppose that $u$ oscillates  on $(r_0,\infty)$. Then  $u(r)$ has infinitely many reversive S-profile on $(r_0,\infty)$.  Let $S_{[\alpha,\beta]}$ be its first  reversive S-profile on $(r_0,\infty)$.
The assumption that $(u+v)_r(r)>0$ on $(r_0,\infty)$ implies the nsd-condition for $(u,v)$ on $(r_0,\infty)$.
Combining the nsd-condition on $(r_0,\infty)$ and Lemma \ref{co1} (b), the local maximum values of  $u(r)$ on $(\alpha,\infty)$ are less than $\log\frac{1}{2}$. Thus,
$$v(r)<u(r)<\log\frac{1}{2}\mbox{ on } (\alpha,\infty). $$
By this and $(u+v)_r(r)>0$ on $(r_0,\infty)$,
\begin{equation}
\begin{split}
\alpha(u+v)_r(\alpha)&\geq \int_{\alpha}^{\infty}s(f_1+f_2)ds\\
&\geq 2e^{(u+v)(\alpha)}\int_{\alpha}^{\infty}sds=\infty,
\end{split}
\end{equation}
which is a contradiction. Hence, $u(r)$ cannot oscillate  on $(r_0,\infty)$.\\
{\it Step 3. }  Suppose that $u(r)$ is increasing  on some interval $(r_1,\infty)\subset (r_0,\infty)$, which implies $\lim_{r\to \infty}u(r)$ exits.
By this and $(u+v)_r>0$ on $(r_0,\infty)$,
$$v(\infty)-v(r_1)>u(\infty)-u(r_1),$$
which implies $v(\infty)>-\infty$. Lemma \ref{remark1}  suggests that $\lim_{r\to \infty}u=\lim_{r\to \infty}v=0$.\\
{\it Step 4. }
Suppose that $u(r)$ is decreasing  on some interval $(r_2,\infty)\subset(r_0,\infty)$. By the nsd-condition for $(u,v)$ on $(r_0,\infty)$, $v$  is increasing  on  $(r_2,\infty)$. Thus, $\lim_{r\to\infty}u(r)$ and $\lim_{r\to\infty}v(r)$ exist and are not both equal to $0$, a contradiction to Lemma \ref{remark1}.\\
{\it Step 5. }{\it  Finally, we show  $v_r(r)>0$ for $r$ sufficiently large.}\par
Since $\lim_{r\to\infty}v=0$, there exists $r_3\in(r_0,\infty)$ such that
$v_r(r_3)>0$ and $u(r)>v(r)>\log\frac{1}{2}$ on $(r_3,\infty)$.
By this, we have $f_1-2f_2<0$ on $(r_3,\infty)$.
Hence,
$$rv_r(r)=r_3v_r(r_3)+\int_{r_3}^rs(f_1-2f_2)ds>0,$$
where $\lim_{r\to\infty}v(r)$ is used.

\end{proof}

\section{ASYMPTOTIC BEHAVIOR (2): GENERAL CASE}
In this section, we consider more general cases.  In the following lemma,  we show that if $u$ and $v$ have only one intersection point $s_2$ on $(s_1,\infty)$  with
 $$v_r(s_1)\leq u_r(s_1),\,\,(u+2v)_r(s_1)\leq 0,\,\,\,u>v\mbox{   on   }(s_1,s_2)\mbox{  and   } v>u\mbox{   on }(s_2,\infty),\,\,\, $$
then $v$ is less than $\log\frac{1}{2}$ on $(s_2,\infty)$.

\begin{lemma}\label{lemma88}
Let $0\leq s_1<s_2$.  Assume that   $s_2$ is a intersection point of $u(r)$ and $v(r)$, with  $u(r)>v(r)$ on $(s_1,s_2)$ and $v(r)>u(r)$ on $(s_2,\infty)$.
We further suppose  that 
$$v_r(s_1)\leq u_r(s_1)\mbox{    and   }(u+2v)_r(s_1)\leq 0.$$
Then  $$u(r) \mbox{  and  }v(r)\mbox{  are  less than   }  \log\frac{1}{2}\mbox{     on    }  [s_2, \infty).$$
\end{lemma}
\begin{proof}
{\it Step 1. }{\it We show $u(r)$ and $v(r)$ satisfy the nsi-condition on $(s_1,\infty]$, and thus   $u_r(s_2)<0$.}\par
Since we suppose $(u+2v)_r(s_1)\leq0$ and $v(r)<u(r)$ on $(s_1,s_2)$,
then \begin{equation}\label{eq0250}
r(u+2v)_r(r)=s_1(u+2v)_r(s_1)-3\int_{s_1}^rsf_2(s)ds<0\mbox{    for   }r\in(s_1,s_2].
\end{equation}
Thus,  $u_r(s_2)<0$.
By this, $(u+2v)_r(s_2)<0$, and $u(r)<v(r)$ on $(s_2,\infty)$, we get
\begin{equation}\label{eq0360}
r(2u+v)_r(r)=s_2(2u+v)_r(s_2)-3\int_{s_2}^rsf_1(s)ds<0\mbox{  for   }r\in[s_2,\infty).
\end{equation}
\\
{\it Step 2. } {\it We  show that $u(s_2)=v(s_2)<\log\frac{1}{2}$.}\par
By \eqref{eq0250}, $u_r(s_2)<0$ and Lemma \ref{co1}, we know that if $u(r)$ has at least one $S$-profile on $(s_1,s_2)$,
 then $u(s_2)$ must be less than $\log\frac{1}{2}$.
Thus,  we need only consider the following cases:\\
\begin{itemize}
\item[(1)]  $u(r)$ is decreasing on $(s_1,s_2)$.
\item[(2)]  $u(r)$ is increasing on $(s_1,r_1)$ and $u(r)$ is decreasing on $(r_1,s_2)$ for some $r_1\in(s_1,s_2)$.  (By \eqref{eq0250},  $v_r(r_1)<0$.)
\end{itemize}
Let $\overline{s}=s_1$ if the first case holds, and $\overline{s}=r_1$ if the second case holds.
Now, assume $u(s_2)=v(s_2)\geq\log \frac{1}{2}$. Then,
$$r(u-v)_r(r)=\overline{s}(u-v)_r(\overline{s})+3\int_{\overline{s}}^r s(f_2-f_1)ds>0\mbox{   on  }(\overline{s},s_2], $$
because $u(r)>v(r)$ on $(\overline{s},s_2)$   and   $u(r)>u(s_2)\geq \log\frac{1}{2}. $
Obviously, it is a contradiction. Therefore, we proved  $u(s_2)=v(s_2)<\log\frac{1}{2}$.\\
\\
{\it Step 3. } {\it We show that $v(r)<\log\frac{1}{2}$ on $[s_2,\infty)$ if $v_r(s_2)\leq 0$.}\par
By $v_r(s_2)\leq 0$, if $v(r)$ attains local maximum at $\beta\in (s_2,\infty)$, then
$v(r)$ has an $S_{[\alpha,\beta]}$-profile where $[\alpha,\beta]\subset[s_2,\infty)$.
Hence, by \eqref{eq0360} and applying $v$ to Lemma \ref{co1} (1),
the local maximum values of $v(r)$  are less than $\log\frac{1}{2}$.
Thus, we only need to  consider the case that $v(r)$ is increasing  on some interval $(R_0,\infty)\subset[s_2,\infty)$.
By \eqref{eq0360}, $u_r(r)<0$ on $(R_0,\infty)$. Lemma \ref{remark1} suggests
$$v(r)\mbox{ is increasing to   }\log\frac{1}{2}\mbox{  and  } u(r)\mbox{  is  decreasing to }-\infty\mbox{  as }r\mbox{  tends to }\infty,$$
i.e., $$v(r)<\log\frac{1}{2}\mbox{   on   }[R_0,\infty).$$
By this and $u(s_2)=v(s_2)<\log\frac{1}{2}$, we have
$$u(r)\leq v(r)<\log\frac{1}{2}\mbox{   on   }[s_2,\infty)\mbox{   provided   }v_r(s_2)\leq 0.$$\\

{\it Step 4. } {\it We show that $v(r)<\log\frac{1}{2}$ on $[s_2,\infty]$ if $v_r(s_2)> 0$.}\par
We need  to consider the following possible cases:
\begin{itemize}
\item[(1)] $v_r(r)>0$ on $[s_2,\infty]$.
\item[(2)] There is $\beta\in(s_2,\infty)$ so that $v(r)$ attains  local maximum at $\beta$ and $v_r(r)\geq 0$ on $[s_2,\beta]$.
\end{itemize}
For the first case, $u_r(r)<0$ on $[s_2,\infty]$ by \eqref{eq0360}. Thus, Lemma \ref{remark1} suggests that
$$\lim_{r\to\infty}v(r)=\log\frac{1}{2}\mbox{  and  }\lim_{r\to\infty}u(r)=-\infty.$$
For the second case, the difference between step 3 and step 4 is that, with $v_r(s_2)>0$,  $v(r)$ attains its first local maximum at $\beta\in(s_2,\infty)$ cannot guarantee
$v(r)$ has an  $S$-profile on $(s_2,\infty)$. However, since $v_r(s_1)\leq 0$ and $v_r(s_2)>0$,  there exits $\alpha\in[s_1,s_2)$ with
 $$v_r(\alpha)=0\mbox{   and   }v_r(r)>0\mbox{  on   }(\alpha,s_2],   $$
which implies
$v(r)$ has an $S_{[\alpha,\beta]}$-profile. By \eqref{eq0360}  and  Lemma \ref{ll} (1),
$$v(\beta)<\log\frac{1}{2}.$$
By this, $v_r(\beta)=0>u_r(\beta)$ and Lemma \ref{lemma7},  we have
$$v(r)<\log\frac{1}{2}\mbox{   on  } [\beta,\infty).$$\\
We conclude  that $u(r)< v(r)<\log\frac{1}{2}\mbox{   on  } [s_2,\infty).$
\end{proof}

\begin{lemma}\label{lemma8}
Let $0\leq s_1<s_2<s_3$.  Assume that   $s_2$ and $s_3$ are consecutive intersection points of $u(r)$ and $v(r)$, with  $u(r)>v(r)$ on $(s_1,s_2)$ and $u(r)<v(r)$ on $(s_2,s_3)$.
\begin{itemize}
\item[(1)] Suppose  that $v_r(s_1)\leq u_r(s_1)$  and $(u+2v)_r(s_1)\leq 0$.
Then  $$u(r) \mbox{  and  }v(r)\mbox{  are  less than   }  \log\frac{1}{2}\mbox{   for   }r\in[s_2, \infty).$$

\item[(2)]  Suppose that  $u(s_1)=v(s_1)$,  $(u+v)_r(r)>0$ on $[s_1,s_3]$,         and $u(s_3)=v(s_3)<\log\frac{1}{2}$.  Then
$$u(r)<\log\frac{1}{2}\mbox{   for   }r\in[s_1,s_2].$$
\end{itemize}
\end{lemma}
\begin{proof}
(1) {\it Step 1. }{\it We show $u(r)$ and $v(r)$ satisfy the nsi-condition on $[s_1,s_3]$, and thus   $u_r(s_2)<0$ and $v_r(s_3)<0$.}\\
{\it Step 2. } {\it We  show that $u(s_2)=v(s_2)<\log\frac{1}{2}$.}\par
We omit the proofs of the steps 1 and 2, because they have exactly the same structure as in the steps 1 and 2 of Lemma \ref{lemma88}.\\
{\it Step 3. } {\it We show that $v(r)<\log\frac{1}{2}$ on $[s_2,s_3]$ if $v_r(s_2)\leq 0$.}\par
Since $v_r(s_2)\leq 0$ and $v_r(s_3)<0$, then either $v$ is decreasing on $(s_2,s_3)$, or
$v$ attains local maximum on $(s_2,s_3)$. The first case implies $v(r)<\log\frac{1}{2}$ on $[s_2,s_3]$.
For the second case, by $v_r(s_2)\leq 0$ and $v_r(s_3)<0$, $v(r)$ attains local maximum at $a\in(r_2,r_3)$ implies
$v(r)$ has an $S_{[b,a]}$-profile on $(r_2,r_3)$. By the step 1 and Lemma \ref{co1}(1),
the local maximum values of $v(r)$ on $[s_2,s_3]$   are less than $\log\frac{1}{2}$. So, we proved $v(r)<\log\frac{1}{2}$ on $[s_2,s_3]$ provided $v_r(s_2)\leq 0$.\\
\\
{\it Step 4. }{\it We show $v(r)<\log\frac{1}{2}$ on $[s_2,s_3]$ if $v_r(s_2)>0$.}\par
The difference between step 3 and step 4 is that, with $v_r(s_2)>0$, $u(r)$ attains its first local maximum on $(s_2,s_3)$ cannot guarantee
$v(r)$ has an  $S$-profile on $(s_2,s_3)$.  Since $v_r(s_1)\leq 0$ and $v_r(s_2)>0$,  there exits $r_1\in(s_1,s_2)$ with
 $$v_r(r_1)=0\mbox{   and   }v_r(r)>0\mbox{  on   }(r_1,s_2].   $$
Let $r_2$ be the first local maximum point of $v$ on $(s_2,s_3)$ which implies $v(r)$ has an $S_{[r_1,r_2]}$-profile. By the step 1  and  Lemma \ref{co1} (1),
$$v(r_2)<\log\frac{1}{2}.$$
Since $v_r(r_2)=0$ and $v_r(s_3)<0$, a maximum point of $v$ on $(r_2,s_3)$ comes with an $S$-profile of $v$ on $(r_2,s_3)$.
Hence, the maximum values of $v(r)$ on $(r_2,s_3)$ are less than $\log\frac{1}{2}$.
As in the step 3, we have
$$v(r)<\log\frac{1}{2}\mbox{   on  } [r_2,r_3].$$\\
{\it Step 5. }{\it  Assume $u(r)$ and $v(r)$ have infinitely many intersection points after $s_3$, say $\{s_i\}_{i=4}^{\infty}$, with $s_i<s_{i+1}$. We  show that  $u(r)$ and $v(r)$ are less than $\log\frac{1}{2}$ on $[s_3,\infty)$.}\par
Note that $u(r)>v(r)$ on $(s_3,s_4)$. By  the step 1 (see the step 1 of Lemma \ref{lemma88}),  we have
\begin{equation}\label{eq091}
(v+2u)_r(s_2)<0\mbox{    and     }u_r(s_2)<v_r(s_2).
\end{equation}
Applying the steps 3 and 4 to $v(r)$ and $u(r)$  on $[s_3,s_4]$, we have
$$u(r)<\log\frac{1}{2}\mbox{   on    }[s_3,s_4].$$
Hence, by repeating above argument, we have
$$u(r)<\log\frac{1}{2}\mbox{   and   }v(r)<\log\frac{1}{2}\mbox{   on   } [s_3,\infty). $$\\
{\it Step 6.} {\it Assume $u(r)$ and $v(r)$ have no intersection point  after $s_3$. }\par
By \eqref{eq091} and Lemma \ref{lemma88}, we have
$$v(r)<u(r)<\log\frac{1}{2}\mbox{   on }[s_3,\infty). $$\\
{\it Step 7. } {\it  Finally, we show that $u(r)$ and $v(r)$ are  less than $\log\frac{1}{2}$ after $s_3$ if $u(r)$ and $v(r)$ have finite intersection points after $s_3$, say $\{s_i\}_{i=4}^n$ with $s_i<s_{i+1}$. Here, we  assume that $u(r)<v(r)$ on $(s_{n-1},s_{n})$ and $v(r)<u(r)$ on $(s_n,\infty)$. }\par
It is clear by the argument in the step  5 that  $u(r)$ an $v(r)$ are less than $\log\frac{1}{2}$ on $[s_3,s_n]$. Also,
one can easily verify
\begin{equation*}
(v+2u)_r(s_{n-1})<0\mbox{   and         }u_r(s_{n-1})<v_r(s_{n-1}).
\end{equation*}
By this and  Lemma \ref{lemma88}, $u(r)$ is less than $\log\frac{1}{2}$ on $[s_n,\infty)$.\\
\\
(2)
 Since we suppose $(u+v)_r(r)>0$ on $[s_1,s_3]$ and $u(s_3)=v(s_3)<\log\frac{1}{2}$, we have
$$u(s_i)=v(s_i)<\log\frac{1}{2},\,\,i=1,2,$$  and
$$u_r(s_1)>0,\,\,\,v_r(s_2)>0,\,\, \mbox{ and }u_r(s_3)>0.$$ \\
{\it Step 1. } {\it We show $u(r)<\log\frac{1}{2}$ on $[s_1,s_2]$ if $u_r(s_2)\geq 0$.}\par
Since $u_r(s_1)>0$ and $u_r(s_2)\geq 0$, then either $u_r(r)$ in increasing  on $(s_1,s_2)$, or $u$ attains local maximum on $(s_1,s_2)$.
The first case implies $u(r)<\log\frac{1}{2}$ on $[s_1,s_2]$ because  $u(s_2)<\log\frac{1}{2}$. For the second case,with $u_r(s_1)>0$ and $u_r(s_2)\geq 0$, $u(r)$ attains local maximum on $(s_1,s_2]$ implies $u$ has a reversive $S$-profile on $(s_1,s_2]$. By $(u+v)_r(r)>0$ on $[s_1,s_3]$  and Lemma \ref{co1}, the maximum values of $u(r)$ on $[s_1,s_3]$ are less than $\log\frac{1}{2}$.    Thus,
$$u<\log\frac{1}{2}\mbox{  on  }[s_1,s_2]\mbox{  if  } u_r(s_2)\geq 0.$$\\
\\
{\it Step 2. }{\it  We show that $u<\log\frac{1}{2}$ on $[s_1,s_2]$ if $u_r(s_2)<0$.}\par
Since we suppose $u_r(s_2)<0$ and  $u_r(s_1)>0$, $u(r)$ attains local maximum on $(s_1,s_2)$. Unlike the step 1, $u(r)$
attains local maximum on $(s_1,s_2)$ does not necessarily imply $u(r)$ has reversive $S$-profile on $(s_1,s_2)$.
Let $u$  attain local maximum at $r_3\in(s_1,s_2)$ and $u_r\leq 0$ on $(r_3,s_2]$.
Since we suppose $u_r(s_2)<0$ and  $u_r(s_3)>0$, there exists $r_4\in(s_2,s_3)$ such that
$$u_r(r)<0\mbox{   on    }[s_2,r_4)\mbox{   and   }u_r(r_4)=0,$$
which implies $u(r)$ has a reversive
$S_{[r_3,r_4]}$-profile. By $(u+v)_r(r)>0$ on $[s_1,s_3]$  and Lemma \ref{ll},
$$u(r_3)<\log\frac{1}{2}.$$
With $u_r(s_1)>0$ and $u_r(r_3)=0$, $u(r)$ attains local maximum on $(s_1,r_3)$ implies that $u(r)$ has a reversive $S$-profile on $(s_1,r_3)$.
By $(u+v)_r(r)>0$ on $[s_1,s_3]$ and Lemma \ref{ll} (2),
the local maximum values of $u(r)$ on $[s_1,s_3]$ are less than $\log\frac{1}{2}$.
It follows that
$$u<\log\frac{1}{2}\mbox{  on  }[s_1,s_2].$$
\end{proof}

\section{PROOF OF THEOREM \ref{thm1}}
Combing Lemma \ref{lemma7}, Lemma \ref{lemma88} and Lemma \ref{lemma8}(1), we have the following lemma which shows if
\begin{equation}\label{condition1}
v(r_1)\leq u(r_1),\,\, v_r(r_1)\leq u_r(r_1)\mbox{   and   }(u+2v)_r(r_1)\leq 0,
\end{equation}
then the behavior of $u(r)$ and $v(r)$  must be case (1) described in the statement of Theorem \ref{thm1}.\\
\begin{lemma}\label{lemma888}
Suppose $u(r)$ and $v(r)$ satisfy the condition \eqref{condition1}.  Then
 $u(r)$ and $v(r)$ are less than $\log\frac{1}{2}$ on some interval $(R_1,\infty)\subseteq (r_0,\infty)$.
\end{lemma}\par
Now, we are ready to prove Theorem \ref{thm1}.\\
{\bf Proof of Theorem \ref{thm1}}\\
{\it Step 1.}. By the equation
\begin{equation}
r(u+v)_r(r)=2(N_1+N_2)-\int_{0}^r s(f_1+f_2)ds,
\end{equation}
we know that
$$\mbox{
either there is  }r_1\mbox{   such that }r_1(u+v)_r(r_1)=0\mbox{  and  }(u+v)_r>0\mbox{  on }[0,r_1),$$
  or
  $$r(u+v)_r(r)> 0\mbox{   for   }r\in(0,\infty).$$
For the first case, there are three possibilities on the derivative of $(u,v)$ at $r_1$. (Here, we assume that $u(r)>v(r)$ on some interval $(r_1,r^*_1)$ ):
\begin{itemize}
\item[(A)] $u_r(r_1)=v_r(r_1)=0$.
\item[(B)] $u_r(r_1)=-v_r(r_1)>0$.
\item[(C)]  $u_r(r_1)=-v_r(r_1)<0$.
\end{itemize}
{\it Step 2.} { We discuss the cases (A) and (B).}\par
For  the cases (A) and (B), it is obvious that
$$ u_r(r_1)\geq 0\geq v_r(r_1), (u+2v)_r(r_1)\leq 0,$$
which satisfies \eqref{condition1}.
Hence, by Lemma \ref{lemma888}, the behaviors of $u(r)$ and $v(r)$ after $r_1$  must be case (1)   described in the statement of this theorem.\\
\\
{\it Step 3.}  We discuss the case (C).\par
Unlike the cases (A) and (B), the case (C) does not have  the condition \eqref{condition1} at $r_1$. We will show $u$ and $v$ satisfy \eqref{condition1} at some point on $(r_1,\infty)$.\par
We define
\begin{equation}\label{s1}
\mathbf{s_1}=sup\{s>r_1|\, u_r<0\mbox{  on   }(r_1,s)\}
\end{equation}
 and  \begin{equation}\label{s2}
\mathbf{s_2}=sup\{s>r_1|\, v_r>0\mbox{  on   }(r_1,s)\}.
\end{equation}
We will use these definitions  throughout the step 3.\\
{\it Step 3.1 }  {\it We consider the case that $(u,v)$ has no intersection point after $r_1$.}\\
{\it  Step 3.1.1. } {\it We  show that   $\mathbf{s_2}\leq \mathbf{s_1}$.}\par
Since $(u+v)_r(r)>0$ on $(0,r_1)$ and $u_r(r_1)<0$, then there exits $r_2\in(0,r_1)$ such that
$$u_r(r_2)=0\mbox{   and    }u_r(r)<0\mbox{  on    }(r_2,r_1].$$
Hence,   $(f_2-2f_1)(r_2)=u_{rr}(r_2)\leq 0$.  By $(u+v)_r(r)>0$ on $(0,r_1)$ and  $u_r(r)<0$ on $(r_2,r_1]$, we have  $v_r(r)>0 $  on $[r_2,r_1]$.
It follows that
$u(r)>v(r)$ on $[r_2,r_1]$ and thus
 \begin{equation}\label{eq002}f_1(r_2)\geq \frac{1}{2}f_2(r_2)>0.
 \end{equation}
Suppose for the sake of contradiction  that  $\mathbf{s_2}> \mathbf{s_1}$. Then
\begin{equation}\label{eq032}
u_r<0<v_r\mbox{  on   }(r_2,\mathbf{s_1}),\mbox{ and }u_r(\mathbf{s_1})=0<v_r(\mathbf{s_1}).
\end{equation}
Note that $f_1(r_2)=e^{u(r_2)}(1-2e^{u(r_2)}+e^{v(r_2)})>0$.
By this and  \eqref{eq032},
\begin{equation}\label{eq012}
1-2e^{u(r)}+e^{v(r)}>0\mbox{  on  }[r_2,\mathbf{s_1}].
\end{equation}
Consequently,  $f_1(r)=e^{u(r)}(1-2e^{u(r)}+e^{v(r)})>0$ on $[r_2,\mathbf{s_1}]$ and
\begin{equation}\label{eq007}
r(u+v)_r(r)=-\int_{r_1}^r s(f_1+f_2)ds<0\mbox{ on }(r_1,\mathbf{s_1}].
\end{equation}
By this,  $(u+v)_r(\mathbf{s_1})<0$, which
is a contradiction to \eqref{eq032}. Hence, we proved $$\mathbf{s_2}\leq \mathbf{s_1}.$$\\
{\it Step 3.1.2.  }
By Lemma \ref{remark1},  $u(r)$ and $v(r)$ cannot be monotone after $r_1$ which implies   $\mathbf{s_2}$ is finite.
Since  we suppose
$$u(r)>v(r)\mbox{   on  }[r_1,\infty)\mbox{   and   }v_r(\mathbf{s_2})=0\geq u_r(\mathbf{s_2}),$$
\begin{equation}\label{eq006}
r(u+2v)_r(r)=\mathbf{s_2}(u+2v)_r(\mathbf{s_2})-3\int_{\mathbf{s_2}}^rsf_2(s)ds <0\mbox{  for }r\in(\mathbf{s_2},\infty).
\end{equation}\par
If  $\mathbf{s_1}$ is finite, by \eqref{eq006},
$$u_r(\mathbf{s_1})=0\geq v_r(\mathbf{s_1})\mbox{   and   }(u+2v)_r(\mathbf{s_1})\leq 0.$$
Thus, by applying Lemma \ref{lemma888} to $u(r)$ and $v(r)$ at $\mathbf{s_1}$, we proved Theorem \ref{thm1}.\par
If  $\mathbf{s_1}=\infty$, we consider the following possible cases:
\begin{itemize}
\item[(1)] $v_r(r)>u_r(r)$ on $(\mathbf{s_2},\infty)$.
\item[(2)] There exists $r^*\in(\mathbf{s_2},\infty)$ such that $v_r(r^*)\leq u_r(r^*)$.
\end{itemize}
For the second case, Theorem \ref{thm1} is proved by applying Lemma \ref{lemma888}  again.
For the first case, if  $\lim_{r\to \infty}u(r)< \log\frac{1}{2}$, then we know there exists $r^{**}$ such that
$$v(r)<u(r)<\log\frac{1}{2}\,\mbox{   for   }r\in(r^{**},\infty), $$
which implies (1) of Theorem \ref{thm1}.
Therefore, we may assume  $\lim_{r\to \infty}u(r)\geq \log\frac{1}{2}$. Note that $\mathbf{s_1}=\infty$ implies $u_r(r)<0$ on $[r_1,\infty)$.
Hence, $$v_r(r)>u_r(r)\mbox{   on  }(r_1,\infty.)$$
It follows that
$$v(\infty)-v(r_1)>u(\infty)-u(r_1),$$
which implies $v(\infty)>-\infty.$
But it yields a contradiction to Lemma \ref{remark1}.\\
\\
{\it Step 3.2. }  {\it  Suppose $u(r)$ and $v(r)$ have at least one intersection point after $r_1$. Here, we assume  $r_{3}$ is  the first intersection point of $u(r)$ and $v(r)$ after $r_1$.}\par
To obtain Theorem \ref{thm1}, we  want to show that
\begin{equation}\label{condition2}
 u_r(r_3)\leq v_r(r_3)\mbox{   and   } (v+2u)_r(r_3)\leq 0.
\end{equation}
Then Theorem \ref{thm1} follows by applying Lemma \ref{lemma888}. Recall the definition of $\mathbf{s_1}$ and $\mathbf{s_2}$ in \eqref{s1} and \eqref{s2}.
By the step 3.1.1., if $\mathbf{s_1}<r_3$,  then $\mathbf{s_2}\leq \mathbf{s_1}.$
Hence, we consider the following possible cases to verify \eqref{condition2}.
\begin{itemize}
\item[(1)] $\mathbf{s_1},\mathbf{s_2}\in[r_3,\infty)$.
\item[(2)] $\mathbf{s_2}\in(r_1, r_3)$ and $\mathbf{s_1}\in[r_3,\infty)$.
\item[(3)] $ \mathbf{s_2}\leq \mathbf{s_1}\in(r_1,r_3)$.\end{itemize}\par
 The case (1) implies that $u_r(r)<0<v_r(r)$ on $[r_1,r_3)$.
As in \eqref{eq007},
$$(u+v)_r(r)<0  \mbox{   on   } (r_1,r_3],$$
which implies $u_r(r_3)<0$ and thus $(v+2u)_r(r_3)<0$.
Then \eqref{condition2} follows.\par
The case (2) implies
$$u_r(\mathbf{s_2})<0=v_r(\mathbf{s_2}).$$
Consequently,
\begin{equation}\label{eq099}
r(u+2v)_r(r)=\mathbf{s_2}(u+2v)_r(\mathbf{s_2})-3\int_{\mathbf{s_2}}^rsf_2(s)ds<0\mbox{   on  }[\mathbf{s_2},r_3].
\end{equation}
Since $(u-v)_r(r_3)<0$,  we have $u_r(r_3)<v_r(r_3)$ and $(u+2v)_r(r_3)<0$ which imply \eqref{condition2}. \par
For the case (3), we have $u_r(\mathbf{s_2})\leq 0=v_r(\mathbf{s_2})$. Hence, as in \eqref{eq099},
$$(u+2v)_r(r)<0\mbox{   on  }(\mathbf{s_2},r_3],$$
which implies \eqref{condition2} again. \\
\\
{\it Step 4.} We now discuss the case:
\begin{equation}\label{e00}
r(u+v)_r(r)=2(N_1+N_2)-\int_{0}^rs(f_1+f_2)ds>0\text{   on   } (0,\infty)
\end{equation}
{\it Step 4.1. } {\it We claim that if $$u(r_4)=v(r_4)\geq \log\frac{1}{2}, \,\,\,v_r(r_4)<u_r(r_4)\mbox{   and  }(u+v)_r>0\mbox{   on }[r_4,\infty),$$ then   $u(r)$ and $v(r)$ do not intersect  after $r_4$. }\par
Suppose $r_5$ is  the first intersection point of $u(r)$ and $v(r)$ after $r_4$. By $(\ref{e00})$ and  $u(r_4)=v(r_4)\geq\log\frac{1}{2}$,
we have
\begin{equation}
\log\frac{1}{2}\leq \frac{1}{2}(u+v)(r_4)< \frac{1}{2}(u+v)(r)\leq u(r)\mbox{   for  }r\in(r_4,r_5],
\end{equation}
which implies $f_2(r)-f_1(r)>0$ on  $(r_4,r_5)$.
Hence,
\begin{equation}
r(u-v)_r(r)=r_4(u-v)_r(r_4)+3\int_{r_4}^rs(f_2-f_1)ds>0\text{  on } (r_4, r_5].
\end{equation}
Obviously, it is a contradiction.\\
{\it Step 4.2 } We discuss the following possible cases.
\begin{itemize}
\item[(1)] $u(r)$ and $v(r)$ have infinitely many intersection points on $(0,\infty)$.
\item[(2)] $u(r)$ and $v(r)$ have  finite intersection points on  $(0,\infty)$.
\end{itemize}
We will show that only the second case is possible. By this and Lemma \ref{lemma7} (2), if $(u+v)_r(r)>0$ on $(0,\infty)$, then
$(u,v)$ is a topological solution and they have  intersection at most finite times.\\
\\
{\it Step 4.2.1.}
By the step 4.1, if $u(r)$ and $v(r)$ have infinitely many intersection points, say $\{t_i\}_{i=1}^{\infty}$, then $u(t_i)=v(t_i)<\log\frac{1}{2}$, $i=1,\cdots,\infty$.
By Lemma \ref{lemma8} (2), $u$ and $v$ are less than $\log\frac{1}{2}$ on $[t_1,\infty)$.\par
By this and \eqref{e00},
\begin{equation}\label{eq035}
\begin{split}
t_1(u+v)_r(t_1)&\geq\int_{t_1}^{\infty} s(f_1+f_2)(s)ds\\
&\geq 2\int_{t_1}^{\infty}se^{(u+v)(s)}ds\geq 2e^{(u+v)(t_1)}\int_{t_1}^{\infty}sds=+\infty,
\end{split}
\end{equation}
which is a contradiction. Hence, $u(r)$  cannot   intersect $v(r)$ infinitely many times.\\
{\it Step 4.2.2.}
If $u(r)$ and $v(r)$ have finite intersection points, say $\{t_i\}_{i=1}^{n}$.
By applying Lemma \ref{lemma7} (2) to $(u(r),v(r))$ on $[t_n,\infty)$,
$(u(r),v(r))$ is a topological solution.\par

\section{CLASSIFICATION OF RADIAL SOLUTIONS }
In this section, we classify the radial solutions according to their boundary conditions at $\infty$. By Theorem \ref{thm1}, we only need to consider the case (1) solutions.   We want to show that $f_1$ and $f_2$ are in $L^1(\mathbb{R}^2)$. We have the following simple observation.

\begin{lemma}\label{simple}
Suppose $(u,v)$  is a case (1) solution in Theorem \ref{thm1}.\\
(1) Then either $\int_0^{\infty}sf_1(s)ds $ and $\int_0^{\infty}sf_2(s)ds $  both are finite or
$\int_0^{\infty}sf_1(s)ds $ and $\int_0^{\infty}sf_2(s)ds $  both are infinite.\\
(2) If $\int_0^{\infty}sf_1(s)ds=\int_0^{\infty}sf_2(s)ds=\infty$, then $\int_{0}^{\infty}se^{(u+v)(s)}ds<+\infty$.
\end{lemma}\par
Next, we introduce the  Pohozaev identity  for \eqref{QWE}.
\begin{lemma}(The Pohozaev identity)
Suppose $(u,v)$ solves \eqref{QWE}. Then $(u,v)$ satisfies
\begin{equation}\label{qwe008}
\begin{split}
&r^2(u_r^2(r)+u_r(r)v_r(r)+v_r^2(r))+3r^2(e^{u(r)}-e^{2u(r)}+e^{u(r)+v(r)} +e^{v(r)}-e^{2v(r)})\\
=&6\int_0^rs\Big( e^{u(s)}-e^{2u(s)}+e^{u(s)+v(s)} +e^{v(s)}-e^{2v(s)} \Big)ds+4(N_1^2+N_1N_2+N_2^2)
\end{split}
\end{equation}
for $r>0$.
\end{lemma}

Using the  Pohozaev identity, we show that $f_1$ and $f_2$ are in $L^1(\mathbb{R}^2)$ when $(u,v)$ is a case (1) solution in Theorem \ref{thm1}.
\begin{lemma}\label{lemma12}
Suppose $(u,v)$  is a case (1) solution in Theorem \ref{thm1}. Then
$$ f_1\mbox{   and   } f_2\mbox{ are  }L^1\mbox{-integrable  in  }\mathbb{R}^2. $$
\end{lemma}
\begin{proof}
{\it Step 1. }
By Theorem \ref{thm1}, there exits $R_0>0$ such that
$$u(r)\mbox{   and   }v(r)\mbox{ both are less than    }\log\frac{1}{2}\mbox{     if       }r>R_0, $$
which implies $f_1(r)>0$ and $f_2(r)>0$ if $r>R_0$. Hence, we only need to show that $$ \int_{R_0}^{\infty}sf_1(s)ds<+\infty \mbox{   and      } \int_{R_0}^{\infty}sf_2(s)ds<+\infty.$$
Suppose for the sake of contradiction that
$$ \int_{R_0}^{\infty}sf_1(s)ds=+\infty \mbox{   and      } \int_{R_0}^{\infty}sf_2(s)ds=+\infty.$$
It implies that
\begin{equation}\label{qwe011}
\int_{R_0}^{\infty}se^{u(s)}ds=+\infty \mbox{   and      } \int_{R_0}^{\infty}se^{v(s)}ds=+\infty
\end{equation}
because of Lemma \ref{simple} (2).
\\

{\it Step 2. }  We denote
\begin{equation*}
w_1(r)=\max(u(r),v(r))\mbox{  and } w_2(r)=\min(u(r),v(r)).
\end{equation*}
and consider the following two cases:
\begin{itemize}
\item[(1)] $\sup_{r\in(R_0,\infty)}w_1(r) <\log\frac{1}{2}$.
\item[(2)] $\sup_{r\in(R_0,\infty)}w_1(r)=\log\frac{1}{2}$.
\end{itemize}
{\it Step 3. } case (1): $\sup_{r\in(R_0,\infty)}w_1(r)<\log\frac{1}{2}$.\par
By this, there exists $\varepsilon>0$ such that
\begin{equation}\label{qwe010}
e^{w_1(r)}<\frac{1}{2}-\varepsilon\,\,\,\,\mbox{      for         }r>R_0.
\end{equation}
Thus, for $r>>R_0$,
\begin{equation}
\begin{split}
&\int_0^{r}s(f_1+f_2)(s)ds\\
=&\int_{R_0}^{r}s(e^{u(s)}-2e^{2u(s)} +e^{v(s)}-2e^{2v(s)})ds +O(1)\\
\geq & 2\varepsilon\int_{R_0}^{r}s(e^{u(s)} +e^{v(s)})ds
\end{split}
\end{equation}
where Lemma \ref{simple}(2) and \eqref{qwe010} are used.
Applying the inequality
$$\frac{3}{4}(a+b)^2\leq a^2+ab +b^2 $$
for $(a,b)=(ru_r(r),rv_r(r))$ to the  Pohozaev identity \eqref{qwe008}, we have
\begin{equation}\label{qwe009}
\begin{split}
&\frac{3}{4}( ru_r(r)+rv_r(r))^2+3r^2(e^{u(r)}-e^{2u(r)}+e^{u(r)+v(r)} +e^{v(r)}-e^{2v(r)})\\
\leq &r^2(u_r^2(r)+u_r(r)v_r(r)+v_r^2(r))+3r^2(e^{u(r)}-e^{2u(r)}+e^{u(r)+v(r)} +e^{v(r)}-e^{2v(r)})\\
=&6\int_0^rs\Big( e^{u(s)}-e^{2u(s)}+e^{u(s)+v(s)} +e^{v(s)}-e^{2v(s)} \Big)ds+4(N_1^2+N_1N_2+N_2^2)
\end{split}
\end{equation}
Note that  $ ru_r(r)+rv_r(r)=2(N_1+N_2)-\int_0^rs(f_1+f_2)(s)ds$.
Thus, for $r>>R_0$,
$$( ru_r(r)+rv_r(r))^2 \geq 4\varepsilon^2 \Big(\int_{R_0}^{r}s(e^{u(s)} +e^{v(s)})ds\Big)^2  $$
Therefore, for $r>>R_0$,  \eqref{qwe009} implies
\begin{equation}
3\varepsilon^2 \Big(\int_{R_0}^{r}s(e^{u(s)} +e^{v(s)})ds\Big)^2\leq 6\Big(\int_{R_0}^{r}s(e^{u(s)} +e^{v(s)})ds\Big)+O(1).
\end{equation}
Due to \eqref{qwe011}, it is a contradiction. \\
{\it Step 4. } case (2): $\sup_{r\in(R_0,\infty)}w_1(r)=\log\frac{1}{2}$.\par
Thus, there exists $\{r_n\}_{n=1}^{\infty}\to+\infty$ such that
$$
w_1(r_n)=\sup_{r\in[R_0,r_n]}w_1(r)\mbox{ tends to }\log\frac{1}{2}\mbox{ as }r_n\mbox{ tends to }+\infty.
$$
Without loss of generality, we might assume that
$$w_1(r_n)=u(r_n).$$
In terms of $w_1(r)$ and $w_2(r)$, the right hand side of the Pohozaev identity \eqref{qwe008} can be written as
$$6\int_0^rs\Big( e^{w_1(s)}-e^{2w_1(s)}+e^{w_1(s)+w_2(s)} +e^{w_2(s)}-e^{2w_2(s)} \Big)ds+4(N_1^2+N_1N_2+N_2^2).$$\par
We estimate $\int_0^rs\Big( e^w_1(s)-e^{2w_1(s)}+e^{w_1(s)+w_2(s)} +e^{w_2(s)}-e^{2w_2(s)} \Big)ds$.
Since we suppose
$r(u+v)_r(r)=2(N_1+N_2)-\int_0^rs(f_1+f_2)(s)ds\to-\infty\mbox{  as    }r\to+\infty, $
there exists $R_1\geq R_0$ such that
$$
(u+v)(r)\leq -6\log r\mbox{   if  }r>R_1.
$$
Denote
$$
\Gamma_1=\{ r\in[R_1,\infty)| w_1(r)\geq -3\log r     \}
$$
and
$$
\Gamma_2=\{ r\in[R_1,\infty)| w_1(r)< -3\log r    \}.
$$
Since $(u+v)(r)\leq-6\log r$ on $[R_1,\infty)$, then
\begin{equation}\label{qwe014}
w_2(r)\leq -3\log r\mbox{  on  }\Gamma_1.
\end{equation}
Without loss of generality, we might assume that  $r_1>>R_1$. Thus,
\begin{equation}
\begin{split}
&\int_0^{r_n} s\Big( e^w_1(s)-e^{2w_1(s)}+e^{w_1(s)+w_2(s)} +e^{w_2(s)}-e^{2w_2(s)} \Big)ds\\
=&\int_{R_1}^{r_n} s\Big( e^w_1(s)-e^{2w_1(s)}+e^{w_1(s)+w_2(s)} +e^{w_2(s)}-e^{2w_2(s)} \Big)ds+O(1)\\
=& \int_{(R_1,r_n) \bigcap\Gamma_1 }\Big( e^w_1(s)-e^{2w_1(s)}+e^{w_1(s)+w_2(s)} +e^{w_2(s)}-e^{2w_2(s)} \Big)ds\\
&+ \int_{(R_1,r_n) \bigcap\Gamma_2}\Big( e^w_1(s)-e^{2w_1(s)}+e^{w_1(s)+w_2(s)} +e^{w_2(s)}-e^{2w_2(s)} \Big)ds+O(1).
\end{split}
\end{equation}
Therefore, we have
\begin{equation}\label{qwe012}
\begin{split}
& \int_{(R_1,r_n) \bigcap\Gamma_1}\Big( e^w_1(s)-e^{2w_1(s)}+e^{w_1(s)+w_2(s)} +e^{w_2(s)}-e^{2w_2(s)} \Big)ds\\
\leq &\Big( e^{u(r_n)}-e^{2u(r_n)}\Big)\int_{R_1}^{r_n}s ds +O(1) \\
=& \frac{1}{2}r_n^2\Big( e^{u(r_n)}-e^{2u(r_n)}\Big)+O(1)
\end{split}
\end{equation}
where $w_2(r)\leq -3\log r$ on $(R_1,r_n) \bigcap\Gamma_1$ is used.
Moreover, one can easily see that
\begin{equation}\label{qwe013}
\int_{(R_1,r_n) \bigcap\Gamma_2}\Big( e^{w_1(s)}-e^{2w_1(s)}+e^{w_1(s)+w_2(s)} +e^{w_2(s)}-e^{2w_2(s)} \Big)ds=O(1).
\end{equation}
Combining \eqref{qwe014}, \eqref{qwe012} and \eqref{qwe013}, \eqref{qwe009} implies that when $r=r_n$,  we have
\begin{equation}
\frac{3}{4}r_n^2(u_r(r_n)+v_r(r_n))^2\leq O(1),
\end{equation}
and then $f_1$ and $f_2$ are in $L^1(\mathbb{R}^2)$.
\end{proof}
{\bf Proof of Theorem \ref{cor1}} \\
By Theorem \ref{thm1}, we show that the case (1) solutions must be either non-topological solutions or mixed-type solutions.
By Lemma \ref{lemma12}, the limit of both $ru_r(r)$ and $rv_r(r)$ exist as $r\to\infty$. By integrating the equations \eqref{QAZ},
it is easy to see that  $\lim_{r\to\infty}(u(r),v(r))$ exists and must be one of the following:
$$
(1)\,\,(\log \frac{1}{2},-\infty)  \hspace{1cm}      (2)\,\,(-\infty,\log\frac{1}{2})       \hspace{1cm}   (3)\,\,(-\infty,-\infty).
$$
By this and Theorem \ref{thm1}, Theorem \ref{cor1} follows.\qed

\section{PROOF OF COROLLARY \ref{cor2}}
In this section, it will be shown that, for entire radial solutions of \eqref{QAZ},  $f_1$ and $f_2$ are in $L^1(\mathbb{R}^2)$. Thus,
we  denote these two quantities
$$ -\beta_1=2N_1+\int_0^{\infty}s(f_2-2f_1)(s)ds       $$
and
$$ -\beta_2=2N_2+\int_0^{\infty}s(f_1-2f_2)(s)ds       $$
to characterize the behaviors of solutions at infinity of \eqref{QAZ}.
It is easy to see that $\lim_{r\to\infty}ru_r(r)=-\beta_1$ and  $\lim_{r\to\infty}ru_r(r)=-\beta_1$  by integrating equations \eqref{QAZ}.\\
\\
{\bf Proof of Corollary \ref{cor2}:}\\
\\
The proof of Corollary \ref{cor2} is long; therefore  we split the proof of Corollary \ref{cor2} into a number of lemmas. See Lemmas \ref{lemma81}, \ref{lemma82} and \ref{lemma83}.\par
We first investigate topological solutions of \eqref{QAZ}. Since $(u,v)\to(0,0)$ as $r\to\infty$ and
$$\Delta(u+v)= (u+v)+O(|u+v|^2) \mbox{  for  } u \mbox{  and  } v \mbox{ small}, $$
by the estimate of elliptic PDE (see Sec. 16  in \cite{LPY}), we obtain:
\begin{lemma}\label{lemma81}
Suppose $(u,v)$ is a topological solution. Then
$(u(r),v(r))\to(0,0)$ exponentially as $r\to\infty$.
\end{lemma}
For the topological solutions, the $L^1$-integrability of $f_1$ and $f_2$ follows immediately.\par

\begin{lemma}\label{lemma82}
Suppose $(u,v)$ is a non-topological solution of \eqref{QAZ}. Then
$$\beta_1>2,\,\,\beta_2>2,$$
and
\begin{equation}\label{e006}
\beta_1^2+\beta_1\beta_2+\beta_2^2-4(N_1^2+N_1N_2+N_2^2)>6\Big(2(N_1+N_2)+(\beta_1+\beta_2)\Big).
\end{equation}
\end{lemma}

\begin{proof}

{\it Step 1.} {\it We first prove the inequality \eqref{e006} with the assumptions $\beta_1>2$ and $\beta_2>2$.}\par
Using $r(u+v)_r(r)=2(N_1+N_2)-\int_0^{r}s(f_1+f_2)(s)ds$, the Pohozaev identity becomes

\begin{equation}\label{e007}
\begin{split}
&r^2(u_r^2(r)+u_r(r)v_r(r)+v_r^2(r))+3r^2(e^{u(r)}-e^{2u(r)}+e^{u(r)+v(r)} +e^{v(r)}-e^{2v(r)})\\
>&6\Big(2(N_1+N_2)-r(u+v)_r(r)\Big)+4(N_1^2+N_1N_2+N_2^2).
\end{split}
\end{equation}
Since  we suppose  $\beta_1>2$ and $\beta_2>2$, for $r$ sufficiently large, we have
$$u(r)=-\beta_1\log r +O(1)\mbox{   and   }v(r)=-\beta_2\log r +O(r).$$
Plugging these into \eqref{e007} and letting $r$ tend to $\infty$, \eqref{e006}  follows.\\
{\it Step 2.}
Since $(u,v)$ is a non-topological solution, there exists $R_1>0$ such that
$$u(r)\mbox{   and   }v(r)\mbox{   are less than   }\log\frac{1}{4} \mbox{   for  }r\geq R_1.$$
By this and Lemma \ref{lemma12}, we obtain
$$\int_{R_1}^{\infty}s\Big(\frac{1}{2}e^{u(s)}+e^{(u+v)(s)}\Big)ds\leq \int_{R_1}^{\infty}sf_1(s)ds<\infty$$
and
$$\int_{R_1}^{\infty}s\Big(\frac{1}{2}e^{v(s)}+e^{(u+v)(s)}\Big)ds\leq \int_{R_1}^{\infty}sf_2(s)ds<\infty.$$
We conclude that
\begin{equation}\label{e008}
e^{u},\,\,e^{v}\mbox{  and   } e^{u+v}\mbox{   are   in   } L^1(\mathbb{R}^2).
\end{equation}
{\it Step 3.} We show that $\beta_1\geq 2$ and $\beta_2\geq 2$.\par
By the symmetry of the equation, we only show that $\beta_1\geq 2$. Suppose for the sake of contradiction that $\beta_1< 2$, then
$$ru_r(r)\geq -2\mbox{ for  }r\mbox{   sufficiently large,}$$
and thus $u(r)\geq -2\log r+O(1)$  for  $r$   sufficiently large. It contradicts to \eqref{e008}.\\
{\it Step 4.} To show $\beta_1>2$  and $\beta_2>2$, we need to exclude the following other possible cases:
\begin{itemize}
\item[(1)] $\beta_1=2$, $\beta_1<\beta_2$.
\item[(2)]  $\beta_2=2$, $\beta_2<\beta_1$.
\item[(3)]  $\beta_1=2$, $\beta_2=2$.
\end{itemize}
The first two cases are symmetric, hence, we only consider the first case.\\
\\
{\it Step 4.1.}  Suppose  $\beta_1=2$  and  $\beta_1<\beta_2$. Then there exists $R_2$ such that
\begin{equation}\label{e009}
u(r)<\log\frac{1}{4}\mbox{   and   } u(r)>v(r)\mbox{  for  }r\geq R_2.
\end{equation}
Recall that
$$ru_r(r)=2N_1+\int_0^rs\big( (f_2-f_1)(s)-f_1(s)        \big)ds. $$
Note that $(f_2-f_1)(r)=g(v(r))-g(u(r))<0$ and $-f_1(r)<0$ on $(R_2,\infty)$
where \eqref{e009} is used. Thus, $ru_r(r)$ is a decreasing function on $(R_2,\infty)$, which implies
$ru_r(r)> -2$ on $(R_2,\infty)$. It makes the $L^1$-integrality of $e^u$ fail.  Hence, this case is impossible.\\
{\it Step 4.2.}  Suppose $\beta_1=2$ and $\beta_2=2$.\par
For $r$ sufficiently large, \eqref{e007} becomes
\begin{equation}
\begin{split}
&12+3r^2\Big(e^{u(r)}-e^{2u(r)}+e^{v(r)}-e^{2v(r)}\Big)\\
> & 6( 2(N_1+N_2)+4 )+4(N_1^2+N_1N_2+N_2^2)+o(1)
\end{split}
\end{equation}
which implies
$$
r\Big(e^{u(r)}+e^{v(r)}\Big)\geq \frac{3}{r}.
$$
It contradicts to \eqref{e008}.\par
We conclude that $\beta_1>2$ and $\beta_2>2$.
\end{proof}

The mixed-type solutions of \eqref{QAZ} will be discussed as follows.
\begin{lemma}\label{lemma83}
Suppose $(u,v)$ is a mixed-type solution. Then
\begin{itemize}
\item[(1)] $\beta_1=0$ and $\beta_2>2$ if $(u,v)\to(\log\frac{1}{2},-\infty)$ as $r\to\infty$. Moreover,
$u\to\log\frac{1}{2}$  as $r\to\infty$.

\item[(2)] $\beta_2=0$ and $\beta_1>2$ if $(u,v)\to(-\infty, \log\frac{1}{2})$ as $r\to\infty$.  Moreover,
$v\to\log\frac{1}{2}$  as $r\to\infty$.
\end{itemize}
\end{lemma}
\begin{proof}
By the symmetry of the equations, we only need to prove the first part.

{\it Step 1.} {\it We show that $\beta_1+\beta_2>2$. }\par
By Theorem \ref{thm1}, $u$ and $v$ are less than $\log\frac{1}{2}$ on $(R_0,\infty)$. Thus,
$r(u+v)_r(r)$ is a decreasing function on $(R_0,\infty)$. If $\beta_1+\beta_2\leq 2$, then
$$r(u+v)_r(r)> -2\mbox{   on    }(R_0,\infty).$$
It follows that $\int_{R_0}^{\infty}se^{(u+v)(s)}ds=\infty$, which is a contradiction to Lemma \ref{lemma12}.\\
{\it Step 2.} {\it We show $\beta_2>2$.}\par
By Lemma \ref{lemma12},  $\beta_1$ and $\beta_2$ exist and are finite.
Since $\lim_{r\to\infty} u(r)=\log\frac{1}{2}$, then $\beta_1=0$.
By the step 1, $\beta_2>2$.\\
\end{proof}

\section{THE STRUCTURE OF NON-TOPOLOGICAL SOLUTIONS }

We denote $(u(r;\alpha_1,\alpha_2), v(r;\alpha_1,\alpha_2)$ be  a solution of \eqref{QWE} with
the initial data
\begin{equation}
\left\{
\begin{array}{l}
u(r ) = 2N_1 \log r +\alpha_1 + o(1)\\
v(r ) = 2N_2 \log r +\alpha_2 + o(1)
\end{array}
\right.
\text{   as  }  r  \to 0^+.
\end{equation}
Recall the region of initial data of the non-topological solutions of \eqref{QWE}.
\begin{equation}
\Omega=\{(\alpha_1,\alpha_2)|\, (u(r;\alpha_1,\alpha_2),v(r;\alpha_1,\alpha_2))\text{ is a  non-topological solution of }(\ref{QWE})\}.
\end{equation}
We prove Theorem \ref{thm5}.\\
\\
{\bf Theorem {\ref{thm5}} }\\
{\it
 $\Omega$ is an open set. Furthermore, if $\alpha=(\alpha_1,\alpha_2)\in\partial\Omega$, then $u(r;\alpha)$ is either a topological solution or a mixed-type solution.
}

\begin{proof}
We shall prove that if $(\alpha_1,\alpha_2)\in \Omega$, then $(u(r;\alpha'_1,\alpha'_2),v(r;\alpha'_1,\alpha'_2))$ is an entire solution of $(\ref{QWE})$ provided $(\alpha'_1,\alpha'_2)$ close to  $(\alpha_1,\alpha_2)$.
 For convenience, we denote $\alpha=(\alpha_1,\alpha_2)$, $\alpha'=(\alpha'_1,\alpha'_2)$, $u(r;\alpha_1,\alpha_2)=u(r;\alpha)$, $v(r;\alpha_1,\alpha_2)=v(r;\alpha)$, $u(r;\alpha'_1,\alpha'_2)=u(r;\alpha')$ and $v(r;\alpha'_1,\alpha'_2)=v(r;\alpha')$.\\
{\it Step 1.} Assume that
\begin{equation}
\begin{split}
u(r;\alpha)=-\beta_1(\alpha)\log r+O(1)\\
v(r;\alpha)=-\beta_2(\alpha)\log r+O(1)\\
\end{split}
\end{equation}
at infinity.
Since $\beta_i(\alpha)>2$, $i=1,2$, we  write
$$\beta_i(\alpha)=(2+\delta_i(\alpha)),\,\,\,i=1,2.$$
Here $\delta_i>0$, $i=1,2$. \par

Let $\delta=\min\{ \delta_1(\alpha),\delta_2(\alpha) \}$.  There exists $r_0>0$, such that for $r\geq r_0$,
\begin{equation}
ru_r(r;\alpha)<-(2+\frac{\delta}{2}),\,\, rv_r(r;\alpha)<-(2+\frac{\delta}{2})
\end{equation}
and
\begin{equation}
e^{u(r;\alpha)+2\log r}<\frac{\delta^2}{192},\,\,e^{v(r;\alpha)+2\log r}<\frac{\delta^2}{192}.
\end{equation}
By the continuity, for  $\alpha'$ close to $\alpha$, one has the followings:
\begin{equation}\label{qwe003}
r_0u_r(r_0;\alpha')<-(2+\frac{\delta}{4}),\,\, r_0v_r(r_0;\alpha')<-(2+\frac{\delta}{4})
\end{equation}
and
\begin{equation}
e^{u(r_0;\alpha')+2\log r_0}<\frac{\delta^2}{96} ,\,\,\,e^{v(r_0;\alpha)+2\log r_0}<\frac{\delta^2}{96}.
\end{equation}
By Theorem \ref{thm1}, we can assume that
\begin{equation}\label{qwe004}
u(r_0, \alpha),\,\, v(r_0, \alpha),\,\, u(r_0, \alpha')\mbox{    and     } v(r_0, \alpha')\mbox{    are   less than  }\log\frac{1}{2}.
\end{equation}

{\it Step 2. }{\it We show that $ru_r(r;\alpha')<-(2+\frac{\delta}{8})$  and   $rv_r(r;\alpha')<-(2+\frac{\delta}{8})$  for
$ r\in(r_0,\infty)$.}\par
We prove it by contradiction. By symmetry,  we can  assume, without loss of generality, that there exists $r_1\in(r_0,\infty)$ such that
\begin{equation}\label{qwe001}
\begin{array}{ll}
(1)& r_1u_r(r_1;\alpha')=-(2+\frac{\delta}{8}).\\
(2)& ru_r(r;\alpha')<-(2+\frac{\delta}{8})\text{  on   }(r_0,r_1).\\
(3)& rv_r(r;\alpha')<-(2+\frac{\delta}{8})\text{  on   }(r_0,r_1).
\end{array}
\end{equation}
We estimate
\begin{equation}\label{QQA}
\begin{split}
&r_1u_r(r_1,\alpha')-r_0u_r(r_0,\alpha')\\
=&\int_{r_0}^{r_1}s\Big[ f_2(u(s;\alpha'),v(s;\alpha')) -2f_1(u(s;\alpha'),v(s;\alpha')) \Big] ds.
\end{split}
\end{equation}
By \eqref{qwe003} and  \eqref{qwe001}, the left side of \eqref{QQA} gives
\begin{equation}\label{PLO}
r_1u_r(r_1;\alpha')-r_0u_r(r_0,\alpha')>\frac{\delta}{8}
\end{equation}
Again, by \eqref{qwe001}, for $r\in [r_0,r_1]$,
\begin{equation}\label{qwe002}
v(r;\alpha')\leq -(2+\frac{\delta}{8})\log r +\Big( v(r_0;\alpha') + (2+\frac{\delta}{8})\log r_0            \Big).
\end{equation}\par
We estimate the right hand side of \eqref{QQA}.
\begin{equation}
\begin{split}
&\int_{r_0}^{r_1}s\Big[ f_2(u(s;\alpha'),v(s;\alpha')) -2f_1(u(s;\alpha'),v(s;\alpha'))        \Big]ds\\
\leq &\int_{r_0}^{r_1}s f_2(u(s;\alpha'),v(s;\alpha')) ds\\
\leq &\frac{3}{2} \int_{r_0}^{r_1}s e^{v(s;\alpha')} ds\\
\leq & \frac{3}{2} \int_{r_0}^{r_1}s e^{-(2+\frac{\delta}{8})\log s + v(r_0;\alpha') + (2+\frac{\delta}{8})\log r_0   }ds\\
= & \frac{3}{2} e^{v(r_0;\alpha') + 2\log r_0 } \frac{8}{\delta}\Big(1-(  \frac{r_0}{r_1} )^{\frac{\delta}{8}} \Big)< \frac{\delta}{8}
\end{split}
\end{equation}
where \eqref{qwe004} and  \eqref{qwe002} are used.
Obviously, it is a contradiction  to \eqref{PLO}.
Thus, $(u(r,\alpha') ,v(r,\alpha'))$ is a non-topological solution.\par
The second part follows obviously from the first part and Theorem \ref{cor1}.

\end{proof}
%
%

\bibliography{nsf}
\bibliographystyle{plain}

\end{document}